\documentclass[twoside,11pt]{article}

\usepackage[abbrvbib]{jmlr2e}
\usepackage{natbib}

\usepackage{microtype}
\usepackage{subfigure}
\usepackage{booktabs} 
\usepackage{bm}
\usepackage{hyperref} 
\hypersetup{colorlinks={true},linkcolor={magenta},citecolor={blue}}
\usepackage{comment}
\usepackage{multirow}
\usepackage[page]{appendix}

\usepackage[flushleft]{threeparttable}
\usepackage{makecell}
\usepackage{enumerate}
\usepackage{diagbox}

\usepackage{amsmath}
\usepackage{mathtools}
\usepackage{algorithm}
\usepackage{algorithmic}

\newtheorem{thm}{Theorem}[section] 
\newtheorem{lem}[thm]{Lemma}

\newtheorem{defi}[thm]{Definition}
\newtheorem{assum}[thm]{Assumption}
\newtheorem{exam}[thm]{Example}

\newtheorem{fact}[thm]{Fact}

\newcommand{\R}{\mathbb{R}}

\newcommand{\be}{\begin{equation}}
	\newcommand{\ee}{\end{equation}}
\newcommand{\en}{\begin{equation*}}
	\newcommand{\een}{\end{equation*}}
\newcommand{\eqn}{\begin{eqnarray}}
	\newcommand{\eeqn}{\end{eqnarray}}
\newcommand{\bmat}{\begin{bmatrix}}
	\newcommand{\emat}{\end{bmatrix}}
\newcommand{\btab}{\begin{tabular}}
	\newcommand{\etab}{\end{tabular}}

\newcommand{\iprod}[2]{\left \langle #1, #2 \right \rangle }

\newcommand{\vct}[1]{\boldsymbol{#1}}

\newcommand{\sgn}{\mathrm{sgn}}

\newcommand{\calG}{\mathcal{G}}

\newcommand{\calM}{\mathcal{M}}
\newcommand{\calN}{\mathcal{N}}
\newcommand{\calO}{\mathcal{O}}
\newcommand{\calP}{\mathcal{P}}

\newcommand{\calR}{\mathcal{R}}
\newcommand{\calS}{\mathcal{S}}

\newcommand{\calU}{\mathcal{U}}

\newcommand{\calX}{\mathcal{X}}

\newcommand{\vx}{\vct{x}}

\firstpageno{1}

\begin{document}

\title{Decentralized Weakly Convex Optimization Over \\ [0.2cm] the Stiefel Manifold}

\author{\name Jinxin Wang$^*$ \email  jxwang@se.cuhk.edu.hk \\
\addr Department of Systems Engineering and Engineering Management\\
The Chinese University of Hong Kong, Hong Kong, China
\AND
\name Jiang Hu$^*$ \email hujiangopt@gmail.com \\
\addr Massachusetts General Hospital and Harvard Medical School\\
Harvard University, Boston, USA 
\AND Shixiang Chen \email chenshxiang@gmail.com \\
\addr JD Explore Academy \\
Beijing, China
\AND Zengde Deng \email dengzengde@gmail.com \\
\addr Cainiao Network  \\
Hangzhou, China
\AND Anthony Man-Cho So \email manchoso@se.cuhk.edu.hk \\
\addr Department of Systems Engineering and Engineering Management\\
The Chinese University of Hong Kong, Hong Kong, China
}

\editor{}
\maketitle
\def\thefootnote{*}\footnotetext{Equal contributions.}

\begin{abstract}
	We focus on a class of non-smooth optimization problems over the Stiefel manifold in the \emph{decentralized} setting, where a connected network of $n$ agents cooperatively minimize a finite-sum objective function with each component being weakly convex in the ambient Euclidean space. Such optimization problems, albeit frequently encountered in applications, are quite challenging due to their non-smoothness and non-convexity. To tackle them, we propose an iterative method called the decentralized Riemannian subgradient method (DRSM). The global convergence and an iteration complexity of $\mathcal{O}(\varepsilon^{-2} \log^2(\varepsilon^{-1}))$ for forcing a natural stationarity measure below $\varepsilon$ are established via the powerful tool of \emph{proximal smoothness} from variational analysis, which could be of independent interest. Besides, we show the local linear convergence of the DRSM using geometrically diminishing stepsizes when the problem at hand further possesses a sharpness property. Numerical experiments are conducted to corroborate our theoretical findings.
\end{abstract}


\section{Introduction}\label{sec:intro}
Decentralized optimization has gained more and more attention during the past decades in various fields ranging from machine learning to control. There are several driving forces behind this. For instance, data storage and manipulation are required to be performed locally for the sake of privacy and security. Besides, it could be computationally prohibitive in modern big data applications if only a single centralized server performs all the computation. Instead, decentralized optimization has a strong potential to utilize many devices with low processing power cooperatively. In this paper, we consider the following problem of weakly convex (possibly non-smooth) optimization over the Stiefel manifold in a decentralized (i.e., multi-agent) manner:
\begin{equation}\label{eq: dwcoptstiefel0}
	\begin{aligned}
		& \min \;f(x) \coloneqq \frac{1}{n} \sum_{i=1}^{n} f_{i}\left(x\right) \\
		& \textrm{ s.t. } \; \ x \in \calM.
	\end{aligned}
\end{equation}
Here, each local component $f_{i}: \R^{d \times r} \rightarrow \R\ (i\in [n]\coloneqq\{1,\dots,n\})$ is assumed to be $\rho$-weakly convex in the ambient Euclidean space $\R^{d \times r}$ (recall that $g(\cdot)$ is $\rho$-weakly convex if $g(\cdot) + \frac{\rho}{2}\| \cdot\|_F^2$ is convex for some constant $\rho \geq 0$) and $\mathcal{M}\coloneqq\operatorname{St}(d, r)=\left\{x \in \R^{d \times r}: d \geq r, x^{\top} x=I_{r}\right\}$ is the Stiefel manifold. The formulation \eqref{eq: dwcoptstiefel0} has found wide applications including sparse principal component analysis \citep{chen2020proximal}, robust principal component analysis \citep[Eq. (69)]{wu2018review}, dual principal component pursuit \citep{zhu2018dual,li2021weakly}, and orthogonal dictionary learning \citep{wang2020unique,li2021weakly,chen2021manifold}.

In the multi-agent setting, there is a connected undirected communication network represented by the graph $\calG$. Each node of $\calG$ corresponds to an agent, and the network of $n$ agents aim to collectively solve \eqref{eq: dwcoptstiefel0}. 
The $i$-th agent holds a local copy $x_i$ of the variable $x$ in \eqref{eq: dwcoptstiefel0}. Let $\mathcal{N}_i$ be the neighborhood of $i$ including itself. For any $i\in [n]$ and $j \in \mathcal{N}_i$, the equality constraint $x_i = x_j$ is required. As $\calG$ is connected, we have the consensus constraint $x_1 = x_2 = \cdots = x_n$. Then, an equivalent reformulation of \eqref{eq: dwcoptstiefel0} is
\begin{equation}\label{eq: dwcoptstiefel}
	\begin{aligned}
		& \min \;f(\vx) \coloneqq \frac{1}{n} \sum_{i=1}^{n} f_{i}\left(x_{i}\right) \\
		& \textrm{ s.t. } \ x_{1}=x_{2}=\cdots=x_{n},\; x_{i} \in \calM, \quad \forall i \in [n],
	\end{aligned}
\end{equation}
where the variable $\vx^\top \coloneqq (x_1^\top, x_2^\top,\dots,x_n^\top) \in \R^{r \times nd}$. 
To obtain the consensual optimal solution to problem \eqref{eq: dwcoptstiefel}, each node of $\mathcal{G}$ needs to mix its local decision variable with its immediate neighbors according to predefined weights. We introduce a mixing matrix $W \in \mathbb{R}^{ n\times n}$ to model the mixing process, whose $(i,j)$-th entry $W_{ij}\ge 0$ represents the weight assigned to node $j$ by node $i$. The following assumption on $W$ is commonly used in decentralized learning.
\begin{assum}
	\label{assum: Connected graph}
	The  mixing matrix $W \in \mathbb{R}^{n \times n}$ is symmetric and doubly stochastic, that is, $W = W^\top, W \ge 0, \sum_{j \in [n]}W_{ij}=1$ for all $i \in [n]$. Moreover, the weight $W_{ij} = 0$ if and only if $j \notin \mathcal{N}_i$.
\end{assum}
An immediate consequence of the Perron-Frobenius theorem \citep{pillai2005perron} is that eigenvalues of $W$ lie in $(-1,1]$. In addition, the second-largest singular value $\sigma_{2}$ of $W$ satisfies $\sigma_{2} \in[0,1)$. 

\subsection{Related work}
Problem \eqref{eq: dwcoptstiefel}, a weakly convex optimization problem over the Stiefel manifold, can be non-smooth and non-convex, rendering it quite challenging to be solved. If the Stiefel manifold constraint is absent in problem \eqref{eq: dwcoptstiefel}, decentralized (sub)-gradient methods were studied in \citet{tsitsiklis1986distributed,nedic2010constrained,yuan2016convergence,zeng2018nonconvex,chen2021distributed} and a distributed dual averaging subgradient method was proposed in \citet{duchi2011dual,liu2022decentralized}. During the past few years, there have been significant efforts in designing provable decentralized algorithms for smooth optimization problems over the Stiefel manifold \citep{chen2021decentralized,ye2021deepca,wang2022decentralized}. Specifically, \citet{chen2021decentralized} developed a decentralized version of the Riemannian gradient method, \citet{ye2021deepca} studied a decentralized power method for solving the distributed principal component analysis problem, and \citet{wang2022decentralized} proposed a decentralized augmented Lagrangian method. The related works are summarized in Table \ref{tab: comparison}. In sharp contrast, the study of the decentralized non-smooth non-convex problem \eqref{eq: dwcoptstiefel} is still in its infancy. Previously, \citet{li2021weakly} established a convergence guarantee for the Riemannian subgradient method when solving the centralized counterpart \eqref{eq: dwcoptstiefel0}. In this work, we extend the results in \citet{li2021weakly} to the decentralized setting. Our technical developments make novel use of the consensus results in \citet{chen2021local} and the powerful tool of proximal smoothness (see Section \ref{sec: gloconv}) from variational analysis. 
\begin{table}[t] 
	\caption{Comparison with related works. ``s.t." means ``subject to", which indicates the feasible region.} \vspace{0.1cm}
	\centering
	\begin{tabular}{|c|c|c|}
		\hline \diagbox{s.t.}{$f_{i}$} & \makecell[c]{ $L$-smooth,\\ 
			non-convex} &\makecell[c]{Non-smooth, \\ 
			weakly convex}\\
		\hline  $\mathbb{R}^{d\times r}$&\makecell[c]{ DGD \\ \citep{zeng2018nonconvex}\\  }
	& \makecell[c]{ DPSM\\ 
		\citep{chen2021distributed}}    \\
	\hline $\mathcal{M}$ & \makecell[c]{DRSGD/DRGTA \\ \citep{chen2021decentralized} \\
		DESTINY \\ \citep{wang2022decentralized}
	}  & \makecell[c]{ \textbf{DRSM}\\ 
		(this paper)}   \\
	\hline
\end{tabular}
\label{tab: comparison}
\end{table}

\subsection{Our contributions}\label{subsec:oc}
In this paper, we propose the \emph{decentralized} Riemannian subgradient method (DRSM) for solving general \emph{weakly convex} optimization problems over the Stiefel manifold of the form \eqref{eq: dwcoptstiefel}. To the best of our knowledge, this is the first work to establish the global convergence and iteration complexity of a decentralized method for solving problem \eqref{eq: dwcoptstiefel}. To establish the convergence of DRSM, we carefully study several properties (particularly the Lipschitzness) of the proximal mapping defined on the Stiefel manifold. We believe that these properties (Lemma \ref{lem:lip}) can also be useful in other applications. Furthermore, we establish the local linear convergence of DRSM in Theorem \ref{thm: sharplinear} using geometrically diminishing stepsizes under the regularity condition of sharpness, thus contributing to existing analysis frameworks for obtaining strong convergence results under various regularity conditions in the decentralized setting \citep{huang2021distributed,tian2019asynchronous,daneshmand2020second,chen2021distributed}.

\section{Preliminaries}\label{sec:preli}
\subsection{Notation}
We use $\otimes$ to denote the Kronecker product. Given a vector $x$, we use $\|x \|_2$ and $\| x\|_1$ to denote its Euclidean norm and $\ell_1$-norm, respectively. Given a matrix $x$, we use $\|x \|_F$ to denote its Frobenius norm and $\|\vx\|_{F,\infty}$ with $\vx^\top \coloneqq (x_1^\top, x_2^\top,\dots,x_n^\top)  \in \R^{r \times nd}$ to denote the norm $\max_{i\in [n]} \|x_i \|_F$. We denote by $1_n$ the $n$-dimensional all-one vector. Given a symmetric matrix $W \in \mathbb{R}^{n \times n}$, we use $\lambda_2(W)$ and $\lambda_n(W)$ to denote its second-largest eigenvalue and smallest eigenvalue, respectively. For a nonempty closed set $\calX$, we use ${\rm dist}(x, \calX)\coloneqq \inf_{y \in \calX} \| x - y\|_F$ to denote the distance between a point $x$ and $\calX$.

\subsection{Subdifferential}
\label{subsec: subdifferential}
For any $i \in [n]$, since $f_i$ is $\rho$-weakly convex, there exists a convex function $g_i: \R^{d\times r} \rightarrow \R$ such that $f_i(x) = g_i(x) - \frac{\rho}{2}\|x\|_F^2$ for any $x \in \R^{d \times r}$. Similar to the convex case, we can define the (Euclidean) subdifferential $\partial f_i$ of $f_i$ as
\begin{align}
\partial f_i(x) = \partial g_i(x) - \rho x, \; \forall x \in \R^{d \times r}.
\end{align}
Since weakly convex functions are subdifferentially regular \citep[Section \uppercase\expandafter{\romannumeral4}]{li2020understanding}, according to the result in \citet[Theorem 5.1]{yang2014optimality}, the Riemannian subdifferential $\partial_{\calR} f_i$ of $f_i$ on the Stiefel manifold $\calM$ is given by
\begin{align}
\partial_{\calR} f_i(x) = \calP_{{\rm T}_{x}\calM}(\partial f_i(x)),\; \forall x \in \calM,
\end{align}
where ${\rm T}_x\calM$ is the tangent space to $\mathcal{M}$ at $x$ and $\calP_{{\rm T}_{x}\calM}(y) = y - \frac{1}{2}x(y^\top x +x^\top y)$ is the projection of $y \in \mathbb{R}^{d \times r}$ onto ${\rm T}_x\mathcal{M}$.

The following Riemannian subgradient inequality is useful in establishing convergence results of subgradient-type methods for minimizing weakly convex problems over the Stiefel manifold.
\begin{lem}[Riemannian subgradient inequality \citep{li2021weakly}] \label{lem: riessubgineq}
For any bounded open convex set $\calU$ that contains $\calM$, there exists a constant $L >0$ such that each $f_i$ is $L$-Lipschitz continuous on $\calU$ and satisfies, for any $x,y \in \calM $ and any $ \tilde{\nabla}_\calR f_i(x) \in \partial_\calR f_i(x)$, that
\begin{align} \label{eq: riessubgineq}
	f_i(y) \geq f_i(x) + \langle \tilde{\nabla}_\calR f_i(x), y-x \rangle - \frac{\rho+L}{2} \| y -x \|_F^2.
\end{align}
\end{lem}
An immediate consequence is that the norm of any Riemannian subgradient of $f_i$ is bounded, i.e., $\| \tilde{\nabla}_\calR f_i(x)\|_F \leq L$ \citep[Theorem 9.13]{rockafellar2009variational}.

\subsection{Stationarity measure}
We denote the Euclidean average of the points $x_1,\dots$, $x_n \in \R^{d \times r}$ as $\hat{x}\coloneqq \frac{1}{n} \sum_{i=1}^n x_i$. Since the feasible set of problem \eqref{eq: dwcoptstiefel} is non-convex, the Euclidean average $\hat{x}$ need not be on $\mathcal{M}$ even if $x_1,\dots$, $x_n \in \calM$. Thus, to measure the consensus error, we appeal to the induced arithmetic mean (IMA) on the Stiefel manifold, which is defined as
\begin{align}
\bar{x} \coloneqq \calP_\calM(\hat{x}) \in \mathop{\arg\min}\limits_{y \in \calM} \sum_{i=1}^n \| y - x_i \|_F^2.
\end{align}
It is worth pointing out that the IMA need not be unique. To measure the degree of stationarity of $\bar{\vx} \coloneqq 1_n \otimes \bar{x}$ for problem \eqref{eq: dwcoptstiefel}, we utilize the following manifold analogs of the Moreau envelope and proximal mapping, which are defined as
\begin{align} \label{eq: moresuenvelop}
\begin{cases}
	f_{\lambda}(x) \coloneqq \min\limits_{y \in \calM}\left\{f(y)+\frac{1}{2 \lambda}\|y - x\|_{F}^{2}\right\}, & x \in \calM, \\ P_{\lambda f}(x) \in \underset{y \in \calM}{\mathop{\arg \min}}\left\{f(y)+\frac{1}{2 \lambda}\|y-x\|_{F}^{2}\right\}, & x \in \calM
\end{cases}
\end{align}
for $\lambda > 0$, respectively. Following the discussion in \citet{li2021weakly}, we note that by the first-order optimality conditions of \eqref{eq: moresuenvelop}, we have
\begin{align} \label{eq: moreausurromeasure}
{\rm dist}(0, \partial_\calR f(P_{\lambda f}(x)))
\leq  \frac{1}{\lambda} \| \calP_{{\rm T}_{P_{\lambda f}(x)}\calM}(P_{\lambda f}(x) - x) \|_F \leq  \frac{1}{\lambda} \| P_{\lambda f}(x) - x\|_F.
\end{align}
In particular, if the term $\frac{1}{\lambda} \| P_{\lambda f}(x) - x\|_F$ in \eqref{eq: moreausurromeasure} is smaller than $\varepsilon$, then $P_{\lambda f}(x)$ is a near-stationary point (due to ${\rm dist}(0, \partial_\calR f(P_{\lambda f}(x))) \leq \varepsilon$) and $x$ is close to the near-stationary point $P_{\lambda f}(x)$ as well. Therefore, it is natural to employ $\frac{1}{\lambda} \| P_{\lambda f}(x) - x\|_F$ as an approximate stationarity measure of $x$. The above discussion motivates us to define the following $\varepsilon$-stationarity for problem \eqref{eq: dwcoptstiefel}.

\begin{defi}[$\varepsilon$-stationarity]
\label{def: epsstationary}
A point $\vx^\top = (x_1^\top, x_2^\top,\dots,x_n^\top) \in \R^{r \times nd}$ is an $\varepsilon$-stationary point of problem \eqref{eq: dwcoptstiefel} if the following two conditions hold:
\begin{align}
	& \frac{1}{n} \sum_{i=1}^n \| x_i - \bar{x}\|_F^2 \leq \varepsilon, \\
	& \frac{1}{\lambda^2} \| P_{\lambda f}(\bar{x}) - \bar{x}\|_F^2 \leq \varepsilon.
\end{align}
\end{defi}
The following property of projection onto the manifold $\calM$ will be used in the theoretical analysis later.
\begin{lem}\label{lem: liplikepro} \citep[Lemma 2]{liu2022unified}
For $x \in \calM$ and $y \in \R^{d \times r}$, we have $\| x - \calP_\calM(y)\|_F \leq 2\| x - y\|_F$.
\end{lem}

\subsection{Consensus on the Stiefel manifold}
\label{subsec: consensusstiefel}
The consensus problem over the Stiefel manifold $\calM$ is to minimize the weighted squared distance among all local variables, which can be formulated as
\begin{equation}\label{eq: consensusstiefel}
\begin{aligned}
	& \min \; \varphi^t(\vx) \coloneqq \frac{1}{4} \sum_{i=1}^{n} \sum_{j=1}^n W^t_{ij} \| x_i - x_j\|_F^2 \\
	& \textrm{ s.t. } \ x_{i} \in \mathcal{M}, \quad \forall i \in [n].
\end{aligned}\tag{C-St}
\end{equation}
Here, $W^t$ denotes the $t$-th power of the doubly stochastic matrix $W$ with $t\geq 1$ being an integer. The Riemannian gradient method DRCS proposed in \citet{chen2021local} for problem \eqref{eq: consensusstiefel} is given by
\begin{align} \label{eq: drcs}
x_{i,k+1} =  \calR_{x_{i,k}}(-\alpha \operatorname{grad} \varphi^t_i(\vx_k))  =\calR_{x_{i,k}} (\alpha \calP_{{\rm T}_{x_{i,k}}\calM} (\sum_{j=1}^n W^t_{ij} x_{j,k} ) ),  \tag{DRCS}
\end{align}
where $\operatorname{grad} \varphi^t_i(\vx_k) \in \mathbb{R}^{d \times r}$ represents the Riemannian gradient of $\varphi^t$ with respect to $x_{i,k}$ and $\calR_{x_{i,k}}(\cdot)$ is a retraction operator. We refer the reader to \citet{absil2009optimization,hu2020brief} for optimization on manifold. In the sequel, we only use the polar decomposition-based retraction to simplify theoretical analysis. Although problem \eqref{eq: consensusstiefel} is non-convex, it has been proved in \citet{markdahl2020high} that DRCS with random initialization will converge almost surely to a consensus solution when $r \leq \frac{2}{3}d -1$. Moreover, it has been shown in \citet{chen2021local} that DRCS converges Q-linearly in a local region. Specifically, define
\begin{align}
\calN &\coloneqq \mathcal{N}_{1} \cap \mathcal{N}_{2},  \label{eq: defN}\\
\calN_1 & \coloneqq\left\{\vx:\|\vx-\bar{\vx}\|_{\mathrm{F}}^{2} \leq n \delta_{1}^{2}\right\}, \\
\calN_2 & \coloneqq \left\{\vx:\|\vx-\bar{\vx}\|_{\mathrm{F}, \infty} \leq \delta_{2}\right\},
\end{align}
where $\delta_1, \delta_2$ satisfy $\delta_1 \leq \frac{1}{5\sqrt{r}}\delta_2$ and $\delta_2 \leq \frac{1}{6}$. We have the following local linear convergence result.

\begin{fact}[Linear convergence rate \citep{chen2021local}]\label{thm:linear_rate_consensus}
Suppose that Assumption \ref{assum: Connected graph} holds. Let the stepsize $\alpha$ satisfy $0<\alpha\leq \bar\alpha \coloneqq\min\{\nu \frac{\Phi}{ L_t}, 1,\frac{1}{M}\}$ and $t\geq \lceil\log_{\sigma_2}(\frac{1}{2\sqrt{n}})\rceil$, where $\nu\in[0,1],$ $\Phi=2-\delta_2^2$, $L_t = 1-\lambda_n(W^t) \in (0,2]$, and $M$ is a finite constant depending on the specific choice of the retraction; see \citet[Lemma 2; Remark 1]{chen2021local} and \citet[Appendix B]{boumal2019global}.
The sequence of iterates $\{\vx_k \}$ generated by \eqref{eq: drcs} achieves consensus with a linear convergence rate if the initialization satisfies 
$\vx_0\in \calN$ defined by \eqref{eq: defN}. That is, we have  $\vx_k\in\calN$ for all $k\geq 0$ and 
\begin{align}\label{linear rate of consensus}			
	\|\vx_{k+1} - \bar{\vx}_{k+1}\|_F
	&\leq \rho_t \|\vx_{k} - \bar{\vx}_{k}\|_F,
\end{align}
where $\bar{\vx}_k \coloneqq 1_n \otimes \left(\mathcal{P}_{\calM}(\frac{1}{n} \sum_{i=1}^n x_{i,k} ) \right)$, $\rho_t \coloneqq \sqrt{1-2(1-\nu) \alpha  \gamma_{t}},$ $\mu_t = 1-\lambda_2(W^t)$, and $\gamma_t = (1- 4r\delta_{1}^2) (1- \frac{\delta_{2}^2}{2}) \mu_t \geq \frac{\mu_t}{2} \geq \frac{1-\sigma_2^t}{2}$. 
\end{fact}
If $\nu=1/2$, we have $\alpha\leq \bar\alpha \coloneqq \min\{\frac{\Phi}{2L_t},1,1/M\}$ and $\rho_t=\sqrt{1-\gamma_t\alpha } <1 .$

\section{Decentralized Riemannian Subgradient Method}
Motivated by the local linear convergence result of consensus on the Riemannian manifold \citep{chen2021local}, our proposed DRSM proceeds as follows. At the $k$-th iteration, it performs a consensus step and then updates the local variable using a Riemannian subgradient direction, i.e.,
for $i\in [n]$,
\begin{align} \label{eq: algiteratedrsm}
x_{i,k+1} & = \calR_{x_{i,k}}(-\alpha \operatorname{grad} \varphi^t_i(\vx_k)  - \beta_k \tilde{\nabla}_\calR f_i(x_{i,k})  )   \nonumber \\
& = \calR_{x_{i,k}} (\alpha \calP_{{\rm T}_{x_{i,k}}\calM}({\footnotesize \sum_{j=1}^n} W_{ij}^t x_{j,k}) - \beta_k \tilde{\nabla}_\calR f_i(x_{i,k}) ),\tag{DRSM}
\end{align}
where $\tilde{\nabla}_\calR f_i(x_{i,k})$ is a Riemannian subgradient of $f_i$ at the point $x_{i,k}$, $\alpha$ and $\beta_k$ are stepsizes to be determined shortly, and $t \geq 1$ is an integer denoting the $t$-th power of the mixing matrix $W$ (i.e., performing multistep consensus). We summarized the algorithm in Algorithm \ref{alg: DRSM}.

\begin{algorithm}[!htbp]
\caption{Decentralized Riemannian Subgradient Method (DRSM) for Solving Problem \eqref{eq: dwcoptstiefel}}  
\begin{algorithmic}[1]  
	\STATE \textbf{Input:} $\vx_0 \in \mathcal{N}$, an integer $t \geq \log _{\sigma_2}\left(\frac{1}{2 \sqrt{n}}\right)$, $0 < \alpha \leq \bar{\alpha}$ with $\bar{\alpha}$ being given in Fact \ref{thm:linear_rate_consensus}.
	\FOR{$k=1,2,\dots$ \{each node $i \in [n]$ in parallel\}}
	\STATE Choose diminishing stepsizes $\beta_k = \calO(1/\sqrt{k})$.
	\STATE Perform the update according to \eqref{eq: algiteratedrsm}.
	\ENDFOR
\end{algorithmic}
\label{alg: DRSM}
\end{algorithm}
One can view the update \eqref{eq: algiteratedrsm} as applying the Riemannian subgradient method to the following penalized version of problem \eqref{eq: dwcoptstiefel}:
\begin{align*}
\min_{x_i \in \calM} \beta_k f(\vx) + \alpha \varphi^t (\vx).
\end{align*}
To gradually approach consensus, we need to increase the effect of $\varphi^t$, or equivalently, decrease the effect of $f$. Therefore, $\beta_k$ should be diminishing. We will formally specify this requirement later.

\section{Global Convergence} \label{sec: gloconv}
In this section, we establish the global convergence of the DRSM. A key ingredient involved is to bound the difference between two values of the proximal mapping $P_{\lambda f}$ in \eqref{eq: moresuenvelop} at two different points. Unlike the convex or weakly convex cases in the Euclidean setting \citep{moreau1965proximite,chen2021distributed}, it is unclear whether the proximal mapping in \eqref{eq: moresuenvelop} is single-valued due to the non-convex constraint, let alone the associated Lipschitzian property. However, we show that the proximal mapping in \eqref{eq: moresuenvelop} is indeed single-valued, despite the existence of the non-convex manifold constraint. Moreover, we show that it is Lipschitz continuous using the so-called proximal smoothness of $\calM$ and the weak convexity of $f$. These results are new and may have further applications in the study of decentralized non-smooth optimization  over a manifold. We begin with the following definition.


\begin{defi}
A closed set $\calX$ is $R$-proximally smooth if the projection mapping $\calP_{\calX}(x)$ is a singleton whenever ${\rm dist}(x,\calX) < R$.
\end{defi}
Let us take the Stiefel manifold $\calM = {\rm St}(d,r)$ as an example. In this case, $R=1$ is the largest possible constant \citep[Proposition 3]{balashov2020error}. For a $R$-proximally smooth set, there are two useful properties \citep{clarke1995proximal,davis2020stochastic}. One is that the projection operator is Lipschitz continuous, that is, for all $x,y \in U_{\calX}(r)$ with $0 <r< R$, it holds that
\begin{align} \label{eq: proximallysmoothsetLips}
\| \calP_\calX(x) - \calP_\calX(y) \| \leq \frac{R}{R-r} \|x - y \|,
\end{align}
where $U_{\calX}(r)$ is the $r$-tube around $\calX$ defined as $U_{\calX}(r)\coloneqq \{y \in \R^{d \times r}:{\rm dist}(y,\calX) < r \}$. Another one is the uniform normal inequality, that is, for any point $x\in \calX$ and a normal vector $v \in N_{\calX}(x)$, the inequality 
\begin{align} \label{eq: uniform-norm-ineq}
\langle v, y-x \rangle \leq \frac{\|v\|}{2R}\cdot \| y-x\|^2
\end{align}
holds for all $y \in \calX$.
A local version of the above normal inequality is shown in \citet[Exercise 13.31(b)]{rockafellar2009variational}.

Based on the proximal smoothness of the Stiefel manifold $\calM$ and the weak convexity of $f$, we show in the following lemma that $P_{\lambda f}$ in \eqref{eq: moresuenvelop} is single-valued and Lipschitz continuous, which turns out to be essential in obtaining the convergence guarantee of \eqref{eq: algiteratedrsm}.
\begin{lem} \label{lem:lip}
Let $f:\R^{d\times r} \rightarrow \R$ be a $\rho$-weakly convex function. Suppose that $f$ is $L$-Lipschitz continuous on some bounded open convex set $\mathcal{U}$ containing the Stiefel manifold $\calM$. Then, for any $\lambda \in (0, (\rho + 3L)^{-1})$, the proximal mapping $P_{\lambda f}$ defined in \eqref{eq: moresuenvelop} is single-valued and Lipschitz continuous with modulus $(1-\lambda (\rho + 3L))^{-1}$ over $\calM$. 
\end{lem}
\begin{proof}
			It follows from \citet[Lemma 3.3]{davis2020stochastic} that $P_{\lambda f}(x)$ is a singleton for all $x \in \calM$. For any two points $x, y \in \calM$, let $u= P_{\lambda f}(x)$ and $v = P_{\lambda f}(y)$. By the optimality of the proximal mapping, we have
	\begin{align} \label{eq: proxlipopt}
		\left\{ \begin{aligned}
			x - u & = \lambda \tilde{\nabla} f(u) + u_N, \\
			y - v & = \lambda \tilde{\nabla} f(v) + v_N, \\
		\end{aligned} \right.         
	\end{align}
	where $\tilde{\nabla} f(u) \in \partial f(u)$, $\tilde{\nabla} f(v) \in  \partial f(v)$, $u_N \in N_\calM(u)$, and $v_N \in N_\calM(v)$. Here, we denote $N_\calM(u), N_\calM(v)$ as normal spaces of $\calM$ at points $u$ and $v$, respectively. It follows
	\begin{align}
		\quad \iprod{x-y}{u-v} - \|u-v\|_F^2
		& = \lambda \iprod{\tilde{\nabla} f(u) - \tilde{\nabla} f(v)}{u-v} + \iprod{u_N - v_N}{u -v} \nonumber \\ 
		& \geq -\lambda \rho\|u -v\|_F^2 - \frac{\|u_N\|_F + \|v_N\|_F}{2} \|u-v\|_F^2, \label{eq: proofLip}
	\end{align}
	where the inequality is due to the weak convexity of $f$ and the uniform normal inequality \eqref{eq: uniform-norm-ineq}. According to \eqref{eq: proxlipopt}, we have
	\begin{align*}
		\|u_N\|_F \leq \|x - u\|_F + \lambda L \leq 3\lambda L,
	\end{align*}
	where the first inequality comes from the fact $\|\tilde{\nabla} f(u) \|_F \leq L$ \citep[Theorem 9.13]{rockafellar2009variational} and the second inequality is from \citet[Lemma 3.2]{davis2020stochastic}. Similarly, we have $\|v_N\|_F \leq 3 \lambda L$.
	Therefore, if $\lambda \in (0, (\rho + 3L)^{-1})$, by \eqref{eq: proofLip} we have
	\begin{equation*}
		\| x - y \|_F \| u - v\|_F - \| u - v\|_F^2
		\geq (-\lambda \rho - 3 \lambda L) \| u-v\|_F^2,
	\end{equation*}
	which yields
	\begin{equation*}
		\|u -v\|_F \leq \frac{1}{1-\lambda(\rho + 3L)} \|x -y\|_F.
	\end{equation*}
	This shows the Lipschitz continuity of the proximal operator $P_{\lambda f}$.
\end{proof}

Let us illustrate the results of Lemma \ref{lem:lip} through the following two examples.
\begin{exam}
Consider a function $f: \mathbb{R}^{d \times r} \rightarrow \mathbb{R}$ of the form
\begin{equation*}
	f(x) = \frac{1}{2}\|x - a\|_F^2,
\end{equation*}
where $a \in \R^{d\times r}$ is given. Note that $f$ is (weakly) convex and Lipschitz continuous over $\calM$. We have
\[ \begin{aligned}
	P_{\lambda f}(x) = \mathop{\arg\min}_{y \in \calM} \;\; \left\{ \frac{1}{2}\|y-a\|_F^2 + \frac{1}{2\lambda} \| y -x \|_F^2 \right\}  = \calP_{\calM}(x + \lambda a) 
\end{aligned}
\]
with $x \in \calM$.
According to \eqref{eq: proximallysmoothsetLips}, for any $\lambda$ satisfying $0< \lambda < \|a\|_F^{-1}$, $P_{\lambda f}(x)$ is Lipschitz continuous with modulus $(1 - \lambda \|a\|_F)^{-1}$.
\end{exam}

\begin{exam}\citep[Proposition 4.6]{luss2013conditional}
Consider a function $f: \mathbb{R}^d \rightarrow \mathbb{R}$ of the form
\begin{equation*}
	f(x) = \| x\|_1.
\end{equation*}
Note that $f$ is (weakly) convex and Lipschitz continuous over $\calM$. For any fixed $x \in \mathbb{R}^{d}$ satisfying $\|x\|_2=1$ and any $\lambda$ satisfying $0 < \lambda < \frac{1}{3\sqrt{d}}$, we have
\begin{align*}
	P_{\lambda f}(x) & = \mathop{\arg\min}_{\| y\|_2=1} \;\; \left\{\| y\|_1 + \frac{1}{2\lambda} \| y -x \|_F^2 \right\} \\
	& = \mathop{\arg\max}_{\| y\|_2=1} \;\; \left\{ y^\top x - \| \lambda y\|_1 \right\} \\
	& = \frac{(|x|-\lambda 1_d)_+ \circ \sgn(x)}{\| (|x|-\lambda 1_d)_+ \circ \sgn(x) \|_2},
\end{align*}
where $|x|$ denotes the vector with the $i$-th entry being $\left|x_i\right|$; $\sgn(x)$ denotes the vector with $i$-th entry being $-1,0,1$ if $x_i<0, x_i=0, x_i>0$, respectively; $x_{+}$ denotes the vector with the $i$-th entry being $\max \left\{x_i, 0\right\}$; and ``$\circ$" represents the element-wise product. According to Lemma \ref{lem:lip}, $P_{\lambda f}(x)$ is Lipschitz continuous with modulus $\left(1 - 3\lambda \sqrt{d}\right)^{-1}$. 
\end{exam}

With Fact \ref{thm:linear_rate_consensus} and Lemma \ref{lem:lip} at our disposal, we are ready to establish the global convergence of DRSM. The proof consists of two parts: The first part guarantees the consensus; the second establishes the convergence and iteration complexity.

It is rather typical in decentralized optimization to adopt diminishing stepsizes for achieving exact consensus if subgradient-type methods are applied. Hence, we need the following assumption on the stepsize.
\begin{assum}[Diminishing stepsizes]
\label{assum: diminishingstepsize}
The stepsizes $\beta_k >0, k\ge1$ are non-increasing and satisfy
\begin{align*}
	\sum_{k=0}^{\infty}\beta_k = \infty, \lim\limits_{k\rightarrow \infty} \beta_k = 0, \lim\limits_{k\rightarrow \infty} \frac{\beta_{k+1}}{\beta_k} =1.
\end{align*}
\end{assum}

To control the consensus error, one useful property is the boundedness of the norm of (Riemannian) subgradient of $f$ over the Stiefel manifold, i.e., $\|\tilde{\nabla}_{\calR}f_i(x) \|_F \leq L$ in Section \ref{subsec: subdifferential}. With such a property at hand, we can establish the following consensus result.
\begin{lem}[Consensus error] \label{thm: consensusbound}
Suppose that Assumption \ref{assum: Connected graph} holds.
Let $\alpha \in (0, \bar{\alpha}]$ for some $\bar{\alpha}\in (0,1]$, $0 < \beta_k \leq \min\{\frac{1-\rho_t}{L} \delta_1, \frac{\alpha \delta_1}{5L} \}$, and $t \geq \lceil \log_{\sigma_2}(\frac{1}{2\sqrt{n}}) \rceil$. If $\vx_0 \in \calN$, then the iterates generated by \eqref{eq: algiteratedrsm} satisfy $\vx_k \in \calN$ for all $k \geq 0$. Moreover, we have 
\begin{align}\label{eq: consensusboundgeneralbeta}
	\quad \| \vx_{k+1} - \bar{\vx}_{k+1} \|_F \leq  \rho_t^{k+1} \|\vx_0 - \bar{\vx}_0 \|_F + \sqrt{n}L \sum_{l=0}^k \rho_t^{k-l}\beta_l,
\end{align}
where $\rho_t \in (0,1)$ is defined in \eqref{linear rate of consensus}.
\end{lem}
\begin{proof}
			The proof is similar to that in \citet[Lemma 4.1]{chen2021distributed}. First, we prove $\vx_k \in \calN$ for all $k\geq 0$ by induction. Suppose $\vx_k \in \calN$, we would like to show $\vx_{k+1} \in \calN$. Since $\beta_k \leq \frac{1-\rho_t}{L}\delta_1$, it follows
	\begin{align}
		\|\vx_{k+1} - \bar{\vx}_{k+1} \|_F  &\leq \rho_t \|\vx_{k} - \bar{\vx}_{k} \|_F + \beta_k \sqrt{n}L \label{eq: consensuserroriter} \\
		& \leq \rho_t \sqrt{n}\delta_1 + \beta_k \sqrt{n}L  \leq \sqrt{n}\delta_1, \nonumber
	\end{align}
	which means that $\vx_{k+1} \in \calN_1$. Next, we shall verify $\vx_{k+1} \in \calN_2$. Utilizing the facts of $\beta_k \leq \frac{\alpha \delta_1}{5L} \leq \frac{\alpha}{2L}$ and $\alpha \leq 1$, we get
	\begin{equation*}
		\left\|{\vx}_{k+1}-{\vx}_{k}\right\|_{\mathrm{F}, \infty} \leq  \max _{i \in[n]}\left\|\alpha \operatorname{grad} \varphi^{t}_i\left(\vx_k \right)\right\|_{\mathrm{F}}+\beta_{k} L \leq 2 \alpha \delta_{2}+\frac{\alpha}{2} \leq 1-\delta_{1}^{2},
	\end{equation*}
	where the second inequality comes from \citet[(C.10)]{chen2021decentralized}.
	
	By \citet[Lemma C.4]{chen2021decentralized}, we obtain
	\begin{align*}
		\left\|\bar{x}_{k}-\bar{x}_{k+1}\right\|_{\mathrm{F}} & \leq \frac{1}{1-2 \delta_{1}^{2}}\left(\frac{2 L_{t}^{2} \alpha+L_{t} \alpha}{n}\left\|{\vx}_{k}-\overline{\vx}_{k}\right\|_{\mathrm{F}}^{2}+ \beta_k L+2 M \beta_{k}^{2}L^2\right) \\
		& \leq \frac{1}{1-2 \delta_{1}^{2}}\left[\left(2 L_{t}^{2} \alpha+L_{t} \alpha\right) \delta_{1}^{2}+\beta_k L+2 M \beta_{k}^{2}L^2\right].
	\end{align*}
	Furthermore, since $L_{t} \leq 2, \beta_{k} \leq \frac{\alpha \delta_{1}}{5L}, \alpha \leq 1 / M$, we get
	\begin{align*}
		\left\|\bar{x}_{k}-\bar{x}_{k+1}\right\|_{\mathrm{F}} \leq \frac{1}{1-2 \delta_{1}^{2}}\left(\frac{252}{25} \alpha \delta_{1}^{2}+\frac{\alpha \delta_{1}}{5}\right) \leq \frac{1}{1-2 \delta_{1}^{2}}\left(\frac{252}{625 r} \alpha \delta_{2}^{2}+\frac{\alpha \delta_{2}}{25 \sqrt{r}} \right),
	\end{align*}
	where the last inequality follows from $\delta_{1} \leq \frac{1}{5 \sqrt{r}} \delta_{2}$. Then, one has
	\begin{align}
		\left\|x_{i, k+1}-\bar{x}_{k+1}\right\|_{\mathrm{F}} 
		& \leq \left\|x_{i, k+1}-\bar{x}_{k}\right\|_{\mathrm{F}}+\left\|\bar{x}_{k}-\bar{x}_{k+1}\right\|_{\mathrm{F}} \nonumber \\
		& \leq \| x_{i, k}- \alpha \operatorname{grad} \varphi_{i}^{t}\left(\vx_{k}\right)-\beta_{k} \tilde{\nabla}_{\calR}f_i(x_{i,k}) -\bar{x}_{k}\left\|_{\mathrm{F}}+\right\| \bar{x}_{k}-\bar{x}_{k+1} \|_{\mathrm{F}}  \nonumber \\
		& \leq \left\|x_{i, k}-\alpha \operatorname{grad} \varphi_{i}^{t}\left(\vx_{k}\right)-\bar{x}_{k}\right\|_{\mathrm{F}}+\frac{\alpha \delta_{1}}{5} +\left\|\bar{x}_{k}-\bar{x}_{k+1}\right\|_{\mathrm{F}} \label{eq: consensusdecom}
	\end{align}
	Now, by the fact
	\begin{equation} \label{eq: gradconsensus}
		\operatorname{grad} \varphi_{i}^{t}(\vx)=x_{i}-\sum_{j=1}^{n} W_{i j}^t x_{j}-\frac{1}{2} x_{i} \sum_{j=1}^{n} W_{i j}^{t}\left(x_{i}-x_{j}\right)^{\top}\left(x_{i}-x_{j}\right),
	\end{equation}
	we have 
	\begin{align*}
		&\left\|x_{i, k}-\alpha \operatorname{grad} \varphi_{i}^{t}\left(\vx_{k}\right)-\bar{x}_{k}\right\|_{\mathrm{F}} \\
		\leq & (1-\alpha) \delta_{2}+\alpha\left\|\hat{x}_{k}-\bar{x}_{k}\right\|_{\mathrm{F}}+\alpha\left\|\sum_{j=1}^{n}\left(W_{i j}^{t}-\frac{1}{n}\right) x_{j, k}\right\|_{\mathrm{F}} +\frac{1}{2}\left\|\alpha \sum_{j=1}^{n} W_{i j}^{t}\left(x_{i, k}-x_{j, k}\right)^{\top}\left(x_{i, k}-x_{j, k}\right)\right\|_{\mathrm{F}} \\
		\leq & (1-\alpha) \delta_{2}+2 \alpha \delta_{1}^{2} \sqrt{r}+\alpha\left\|\sum_{j=1}^{n}\left(W_{i j}^{t}-\frac{1}{n}\right) x_{j, k}\right\|_{\mathrm{F}}+2 \alpha \delta_{2}^{2} \\
		\leq &\left(1-\frac{\alpha}{2}\right) \delta_{2}+2 \alpha \delta_{1}^{2} \sqrt{r}+2 \alpha \delta_{2}^{2},
	\end{align*}
	where the first inequality follows from $\alpha \in[0,1]$, the second inequality holds by \citet[Lemma C.1]{chen2021decentralized}, and the last inequality follows from \citet[Lemma D.1]{chen2021decentralized}. Substitute this bound into \eqref{eq: consensusdecom} yields
	\begin{align}
		\left\|x_{i, k+1}-\bar{x}_{k+1}\right\|_{\mathrm{F}} 
		\leq \left(1-\frac{\alpha}{2}\right) \delta_{2}+2 \alpha \delta_{1}^{2} \sqrt{r}+2 \alpha \delta_{2}^{2}+\frac{1}{5} \alpha \delta_{1}+\frac{1}{1-2 \delta_{1}^{2}}\left(\frac{252}{625 r} \alpha \delta_{2}^{2}+\frac{\alpha \delta_{2}}{25 \sqrt{r}} \right) . \label{eq: consensuserrorlast}
	\end{align}
	Therefore, substituting the conditions on $\delta_{1}, \delta_{2}$ in Section \ref{subsec: consensusstiefel} into \eqref{eq: consensuserrorlast} yields
	\begin{align*}
		\left\|x_{i, k+1}-\bar{x}_{k+1}\right\|_{\mathrm{F}} \leq \delta_{2} .
	\end{align*}
	We already obtain $\vx_{k+1} \in \calN_2$.
	Last, unfolding \eqref{eq: consensuserroriter} yields
	\begin{align*}
		\|\vx_{k+1} - \bar{\vx}_{k+1} \|_F  \leq & \rho_t \|\vx_{k} - \bar{\vx}_{k} \|_F + \beta_k \sqrt{n}L
		\leq \rho_t^{k+1} \|\vx_0 - \bar{\vx}_0 \|_F + \sqrt{n}L \sum_{l=0}^k \rho_t^{k-l}\beta_l.
	\end{align*}
	The proof is complete.
\end{proof}

An immediate result by \citet[Lemma  \uppercase\expandafter{\romannumeral2}.4]{chen2021distributed} is that the rate of consensus satisfies $\| \vx_k - \bar{\vx}_k\|_F^2  = \calO(\beta_k^2)$ if $\beta_k = \calO(\frac{1}{k^p})$ with $p \in (0,1]$. Moreover, we have $\| \vx_{k+1} - \bar{\vx}_{k+1} \|_F = \calO(\frac{\sqrt{n}L\beta_k}{1-\rho_t})$.


Next, we establish the convergence of the Moreau envelope sequence $\{f_\lambda(\bar{x}_k)\}_{k=1,2,\dots}$. After that, we prove that the infimum of the defined measure of stationarity vanishes.
\begin{thm}[Iteration complexity] \label{thm: iteration complexity}
Let $0 < \lambda < \min \{\frac{1}{\rho+3L}, \frac{1}{2(\rho + L)} \}$ and $\{x_{i,k}\}$ be the iterative sequence generated by \eqref{eq: algiteratedrsm}. Under Assumption \ref{assum: Connected graph} and Assumption \ref{assum: diminishingstepsize}, the following two results hold:
\begin{enumerate}[(i)]
	\item If $\sum_{k=0}^\infty \beta_k^2 < \infty$, then there exists a constant $\bar{f}_\lambda$ such that \begin{align*}
		\lim\limits_{k \rightarrow \infty} f_\lambda(x_{i,k}) = \lim\limits_{k \rightarrow \infty} f_\lambda(\bar{x}_k) = \bar{f}_\lambda;
	\end{align*}
	\item There exists a sequence $\{a_k\}$ such that $a_k = \calO\left(\frac{L^2 \beta_k^2}{(1-\rho_t)^2}\right)$ for $k\ge0$ and 
	\begin{align*}
		& \inf_{k\ge 0} \| \bar{x}_k - P_{\lambda f}(\bar{x}_k)\|_F^2 \\
		\leq & \frac{ \lambda \left(\bar{f}_\lambda(\vx_0) - \min\limits_{y \in \mathcal{M}} f_\lambda(y) + \sum_{k=0}^\infty a_k + \sum_{k=0}^\infty \frac{\beta_k^2 L^2}{\lambda} \right)}{(\frac{1}{2\lambda} - \rho - L) \sum_{k=0}^\infty  \beta_k},
	\end{align*}
	where $\bar{f}_\lambda(\vx_{0}) \coloneqq \frac{1}{n} \sum_{i=1}^n f_\lambda(x_{i,0})$.
\end{enumerate}
\end{thm}
\begin{proof}
		According to the definition of Moreau envelope, we have
	\begin{align}\label{eq: moreauenvxikp1}
		f_{\lambda}({x}_{i,k+1}) & \leq  f(P_{\lambda f}({x}_{i,k})) + \frac{1}{2\lambda} \| P_{\lambda f}({x}_{i, k}) - {x}_{i, k+1}\|_F^2.
	\end{align}
	Next, we bound the term $\| P_{\lambda f}({x}_{i,k}) - {x}_{i,k+1}\|_F^2$ as follows,
	\begin{align}
		& \quad\;  \| P_{\lambda f}({x}_{i,k}) - {x}_{i,k+1}\|_F^2  \nonumber \\
		& =  \|P_{\lambda f}({x}_{i,k}) - \calR_{x_{i,k}} \left(\alpha {\rm grad} \varphi^t_i(\vx_k) - \beta_k \tilde{\nabla}_\calR f_i(x_{i,k}) \right) \|_F^2 \nonumber \\
		& \leq   \| x_{i,k} + \alpha {\rm grad} \varphi^t_i(\vx_k) - \beta_k \tilde{\nabla}_\calR f_i(x_{i,k}) - P_{\lambda f}({x}_{i,k})  \|_F^2 \nonumber \\
		& =  \| \alpha {\rm grad} \varphi^t_i(\vx_k) - \beta_k \tilde{\nabla}_\calR f_i(x_{i,k}) \|_F^2 + 2 \langle  x_{i,k} - P_{\lambda f}({x}_{i,k}), \alpha {\rm grad} \varphi^t_i(\vx_k) \rangle \nonumber\\
		& \quad  + 2 \beta_k \langle   P_{\lambda f}({x}_{i,k}) - x_{i,k},  \tilde{\nabla}_\calR f_i(x_{i,k}) \rangle + \|  x_{i,k} - P_{\lambda f}({x}_{i,k}) \|_F^2 \nonumber \\
		& \leq     \| \alpha {\rm grad} \varphi^t_i(\vx_k) - \beta_k \tilde{\nabla}_\calR f_i(x_{i,k}) \|_F^2 + 2 \langle  x_{i,k} - P_{\lambda f}({x}_{i,k}), \alpha {\rm grad} \varphi^t_i(\vx_k) \rangle \label{eq: bdproxikp1} \\
		& \quad  + 2 \beta_k \left(f_i(P_{\lambda f}({x}_{i,k})) - f_i(x_{i,k}) + \dfrac{\rho + L}{2} \| x_{i,k} - P_{\lambda f}({x}_{i,k})\|^2_F \right) +\|  x_{i,k} - P_{\lambda f}({x}_{i,k}) \|_F^2, \nonumber
	\end{align}
	where the first inequality is due to properties of the polar retraction in \citet[Lemma 1]{li2021weakly}, and the second inequality is from the Riemannian subgradient inequality in Lemma \ref{lem: riessubgineq}.
	
	Now, we provide bounds for the above terms, respectively.
	
	\noindent \textbf{Part 1}: $\| \alpha {\rm grad} \varphi^t_i(x_k) - \beta_k \tilde{\nabla}_\calR f_i(x_{i,k}) \|_F^2 + 2 \langle  x_{i,k} - P_{\lambda f}({x}_{i,k}), \alpha {\rm grad} \varphi^t_i(x_k) \rangle$.
	\begin{align}
		\sum_{i=1}^n \| \alpha {\rm grad} \varphi^t_i(\vx_k) - \beta_k \tilde{\nabla}_\calR f_i(x_{i,k}) \|_F^2
		& \leq 2 \alpha^2 \sum_{i=1}^n\|{\rm grad} \varphi^t_i(\vx_k) \|_F^2 + 2n \beta_k^2 L^2  \nonumber \\
		& \leq  2 \alpha^2 L_t^2 \| \vx_k - \bar{\vx}_k  \|_F^2 + 2n \beta_k^2 L^2, \label{eq: compbound1}
	\end{align}
	where the second inequality is due to \citet[(C.9)]{chen2021decentralized}.
	\begin{align}
		& \quad 2\alpha \sum_{i=1}^n \left\langle  x_{i,k} - P_{\lambda f}({x}_{i,k}),  {\rm grad} \varphi^t_i(\vx_k) \right\rangle \nonumber \\
		& = 2\alpha \sum_{i=1}^n \left\langle \bar{x}_k - P_{\lambda f}(\bar{x}_k), {\rm grad} \varphi^t_i(\vx_k) \right \rangle + 2\alpha \sum_{i=1}^n \left\langle x_{i,k} - \bar{x}_k + P_{\lambda f}(\bar{x}_k) - P_{\lambda f}(x_{i,k}), {\rm grad} \varphi^t_i(\vx_k) \right\rangle \nonumber \\
		& \leq 4\alpha \sqrt{r} \left\|\sum_{i=1}^n {\rm grad} \varphi^t_i(\vx_k) \right\|_F + \alpha  \sum_{i=1}^n \| x_{i,k} - \bar{x}_k + P_{\lambda f}(\bar{x}_k) - P_{\lambda f}(x_{i,k}) \|_F^2 + \alpha \sum_{i=1}^n \|{\rm grad} \varphi^t_i(\vx_k)  \|_F^2 \nonumber\\
		& \leq 4\alpha \sqrt{r}L_t \| \vx_k - \bar{\vx}_k \|_F^2 + 2 \alpha(1+ \hat{L}^2) \|\vx_k - \bar{\vx}_k \|_F^2 +\alpha L_t^2  \| \vx_k - \bar{\vx}_k \|_F^2,
		\label{eq: difmorgradlip}
	\end{align}
	where we use $\|\bar{x}_k - P_{\lambda f}(\bar{x}_k)  \|_F \leq 2 \sqrt{r}$ in the first inequality, the second inequality is due to \citet[(C.8), (C.9)]{chen2021decentralized}, and $\hat{L}\coloneqq (1 - \lambda(\rho + 3L))^{-1}$.

	\noindent \textbf{Part 2}: $f_i(P_{\lambda f}({x}_{i,k})) - f_i(x_{i,k})+ \frac{\rho + L}{2} \| x_{i,k} - P_{\lambda f}({x}_{i,k})\|^2_F$.
	
	\begin{align*}
		f_i(P_{\lambda f}({x}_{i,k})) - f_i(x_{i,k})
		& = f_i(P_{\lambda f}(x_{i,k})) - f_i(P_{\lambda f}(\bar{x}_k)) + f_i(P_{\lambda f}(\bar{x}_k)) - f_i(\bar{x}_k) + f_i(\bar{x}_k) - f_i(x_{i,k}) \\
		& \leq  L\|P_{\lambda f}(x_{i,k}) - P_{\lambda f}(\bar{x}_k) \|_F + f_i(P_{\lambda f}(\bar{x}_k)) - f_i(\bar{x}_k) + L \| x_{i,k} - \bar{x}_k\|_F \\
		& \leq L(\hat{L}+1)\|x_{i,k} - \bar{x}_k \|_F +  f_i(P_{\lambda f}(\bar{x}_k)) - f_i(\bar{x}_k),
	\end{align*}
	and 
	\begin{align*}
		\frac{\rho + L}{2} \| x_{i,k} - P_{\lambda f}({x}_{i,k})\|^2_F
		& \leq \frac{\rho + L}{2} \| x_{i,k}-\bar{x}_k + \bar{x}_k -P_{\lambda f}(\bar{x}_k) + P_{\lambda f}(\bar{x}_k) - P_{\lambda f}({x}_{i,k})\|^2_F \\
		& \leq (\rho + L)\| \bar{x}_k -P_{\lambda f}(\bar{x}_k)\|_F^2 + (\rho + L)\| x_{i,k}-\bar{x}_k + P_{\lambda f}(\bar{x}_k) - P_{\lambda f}({x}_{i,k})\|^2_F \\
		& \leq  (\rho + L)\| \bar{x}_k -P_{\lambda f}(\bar{x}_k)\|_F^2 + 2(\rho +L)(1+ \hat{L}^2)\|x_{i,k} - \bar{x}_k\|_F^2.
	\end{align*}
	Summing the above two inequalities for $i \in [n]$ gives
	\begin{align*}
		& \quad \sum_{i=1}^n \left(f_i(P_{\lambda f}({x}_{i,k})) - f_i(x_{i,k}) + \frac{\rho + L}{2} \| x_{i,k} - P_{\lambda f}({x}_{i,k})\|^2_F \right)\\
		& \leq L(\hat{L}+1)\sum_{i=1}^n \|x_{i,k}-\bar{x}_k \|_F + 2(\rho +L)(1+ \hat{L}^2) \sum_{i=1}^n \|x_{i,k} - \bar{x}_k\|_F^2 \\
		&\quad  + n\left(f(P_{\lambda f}(\bar{x}_k)) - f(\bar{x}_k) + (\rho+L)\| \bar{x}_k -P_{\lambda f}(\bar{x}_k)\|_F^2 \right).
	\end{align*}
	From the definition of $P_{\lambda f}(\bar{x}_k)$, if $0 < \lambda < \frac{1}{2(\rho+L)}$, we have 
	\begin{align*}
		& \quad \left(f(P_{\lambda f}(\bar{x}_k)) - f(\bar{x}_k) + (\rho+L)\| \bar{x}_k -P_{\lambda f}(\bar{x}_k)\|_F^2 \right) \\
		& = \left(f(P_{\lambda f}(\bar{x}_k)) + \frac{1}{2\lambda}\|P_{\lambda f}(\bar{x}_k) - \bar{x}_k \|_F^2 - f(\bar{x}_k) + (\rho+L-\frac{1}{2\lambda})\| \bar{x}_k -P_{\lambda f}(\bar{x}_k)\|_F^2 \right) \\
		& \leq (\rho+L-\frac{1}{2\lambda})\| \bar{x}_k -P_{\lambda f}(\bar{x}_k)\|_F^2.
	\end{align*}
Denote $\bar{f}_\lambda(\vx_{k+1}) \coloneqq \frac{1}{n} \sum_{i=1}^n f_\lambda(x_{i,k+1})$. Combining the above bounds with \eqref{eq: moreauenvxikp1} as well as \eqref{eq: bdproxikp1} yields
	\begin{align} \label{eq: envelopeappoxescent}
		\bar{f}_\lambda(\vx_{k+1}) - 	\bar{f}_\lambda(\vx_{k}) \leq a_k + \frac{\beta_k^2 L^2}{\lambda}+ \frac{\beta_k(\rho+L - \frac{1}{2\lambda})}{\lambda} \| \bar{x}_k - P_{\lambda f}(\bar{x}_k)\|_F^2,
	\end{align}
	with
	\begin{align*}
		a_k 
		=&   \frac{1}{2n\lambda}\left(2\alpha^2 L_t^2 + 4\alpha \sqrt{r}L_t + 2\alpha (1+\hat{L}^2) + \alpha L_t^2 + 4\beta_k(\rho+L)(1+\hat{L}^2) \right) \|\vx_{k} - \bar{\vx}_k\|_F^2 \\
		& + \frac{\beta_k}{n\lambda}L(\hat{L}+1) \sum_{i=1}^n\|x_{i,k} - \bar{x}_k\|_F.
	\end{align*}
	The first term of $a_k$ is $\calO(\frac{nL^2\beta_k^2}{(1-\rho_t)^2})$, and the second term of $a_k$ is $ \calO(\frac{nL\beta_k^2}{1-\rho_t} )$. Because $f(x)$ is lower bounded on $\calM$, we have $f_\lambda(x)$ is also lower bounded on $\calM$. It follows that
	\begin{align}
		\bar{f}_\lambda(x_{k+1}) - \min_{x\in \calM} f_\lambda(x) \leq  \bar{f}_\lambda(x_{k}) - \min_{x\in \calM} f_{\lambda}(x)  + \calO(\beta_k^2).	
	\end{align}
	By using \citet[Lemma 2]{polyak1987introduction} and $\sum_{k=0}^\infty \beta_k^2 < \infty$, we have $\{ \bar{f}_\lambda(x_k)\}$ converges to some value $\bar{f}_\lambda$.
	
Note that $f_\lambda(x)$ is continuous. Since $\| x_{i,k} - \bar{x}_k \|_F \rightarrow 0$, we have 
	\begin{align*}
		| f_\lambda(x_{i,k}) - f_\lambda(\bar{x}_k) | \rightarrow 0
	\end{align*}
	and 
	\begin{align*}
		| \bar{f}_\lambda(\vx_k) - f_\lambda(\bar{x}_k)|^2 & = |\frac{1}{n}\sum_{i=1}^n f_\lambda(x_{i,k})  - f_\lambda(\bar{x}_k)|^2  \leq \frac{1}{n}\sum_{i=1}^n |f_\lambda(x_{i,k}) - f_\lambda(\bar{x}_k)|^2 \rightarrow 0.
	\end{align*}
	Rearranging the terms in the inequality \eqref{eq: envelopeappoxescent} yields
	\begin{align} \label{eq: complexderive1}
		\frac{\beta_k}{\lambda}(\frac{1}{2\lambda} - \rho - L) \| \bar{x}_k - P_{\lambda f}(\bar{x}_k)\|_F^2 \leq \bar{f}_\lambda(\vx_k) - \bar{f}_\lambda(\vx_{k+1}) + a_k + \frac{\beta_k^2 L^2}{\lambda}.
	\end{align}
	Summing \eqref{eq: complexderive1} over $k=0,1,\dots$ gives
	\begin{align}
		\sum_{k=0}^\infty  \frac{\beta_k}{\lambda}(\frac{1}{2\lambda} - \rho - L) \| \bar{x}_k - P_{\lambda f}(\bar{x}_k)\|_F^2 \leq \bar{f}_\lambda(\vx_0) - \min_{y \in \mathcal{M}} f_\lambda(y) + \sum_{k=0}^\infty a_k + \sum_{k=0}^\infty \frac{\beta_k^2 L^2}{\lambda}.
	\end{align}
	By dividing both sides by $ \sum_{k=0}^\infty  \frac{\beta_k}{\lambda}(\frac{1}{2\lambda} - \rho - L)$, we obtain
	\begin{align*}
		\inf_{k=0,1,\dots} \| \bar{x}_k - P_{\lambda f}(\bar{x}_k)\|_F^2 \leq \frac{\lambda}{(\frac{1}{2\lambda} - \rho - L)} \frac{\bar{f}_\lambda(\vx_0) - \min_{y \in \mathcal{M}} f_\lambda(y) + \sum_{k=0}^\infty a_k + \sum_{k=0}^\infty \frac{\beta_k^2 L^2}{\lambda}}{\sum_{k=0}^\infty  \beta_k}.
	\end{align*}
 This finishes the proof.
\end{proof}

\begin{remark}
According to the proof of Theorem~\ref{thm: iteration complexity}.(ii), if $\beta_k = \calO(1/\sqrt{k})$, for sufficiently large $K$, we have 
\begin{align*}
	\quad \min_{1\leq k\leq K} \| \bar{x}_k - P_{\lambda f}(\bar{x}_k)\|_F^2
	= \calO\left(\frac{\bar{f}_\lambda(\vx_0) - \min\limits_{y \in \mathcal{M}} f_\lambda(y)}{\sqrt{K}} + \frac{L^2}{(1-\rho_t)^2}\frac{\log K}{\sqrt{K}}  \right).
\end{align*}
The above result indicates that the iteration complexity of DRSM is $\mathcal{O}(\varepsilon^{-2} \log^2(\varepsilon^{-1}))$.
\end{remark}

\section{Local Linear Convergence Under Sharpness} \label{sec: locconv}
In this section, we aim at deriving stronger convergence guarantee for DRSM when applied to problem \eqref{eq: dwcoptstiefel} with certain additional structures besides just weak convexity. In the centralized setting, to establish the (local) linear convergence rate of iterative methods for non-convex problems or convex but not strongly-convex problems, certain regularity conditions (e.g., the local error bound condition \citep{zhou2017unified,liu2017estimation,yue2019family,zhu2021orthogonal}, the local Kurdyka-\L ojasiewicz inequality \citep{liu2019quadratic,wang2021linear,li2021convergence}, or the sharpness property \citep{davis2018subgradient,li2020nonconvex,liu2020nonconvex,li2021weakly}) are usually required. In addition, there have been many attempts to establish strong convergence results in stochastic or decentralized settings under the aforementioned regularity properties; see, e.g., \citet{so2017non,chen2021distributed,tian2019asynchronous}. Motivated by such a line of research, we show that if problem \eqref{eq: dwcoptstiefel0} additionally possesses the following sharpness property \citep{burke1993weak,karkhaneei2019nonconvex,li2011weak}, then with geometrically diminishing stepsizes (i.e., $\beta_k = \mu_0 \gamma^k$ with $\mu_0 >0$ and $\gamma \in (0,1)$), our proposed DRSM for problem \eqref{eq: dwcoptstiefel} would converge in a linear rate, provided that it is initialized with a suitable point. 
\begin{defi}[Sharpness]
A set $\calX \subseteq \calM$ is called a set of weak sharp minima for the function $f: \R^{d \times r} \rightarrow \R$ with parameter $\kappa > 0$ if there exists a constant $B>0$ such that for any $x \in U_\calX(B) \cap \calM$, we have
\begin{align} \label{eq: sharpdef1}
	f(x)-f(y) \geq \kappa\ {\rm dist}(x,\calX)  
\end{align}
for all $y \in \calX$.
\end{defi}

Note that if $\calX$ is a set of weak sharp minima for $f$, then it is the set of minimizers of $f$ over $U_\calX(B) \cap \calM$. In addition, when $f$ is continuous (e.g., if $f$ is weakly convex; see Section \ref{subsec: subdifferential}), then $\calX$ can be chosen as a closed set. 

We first estimate the deviation from the mean $\|\vx_k - \bar{\vx}_k \|_F$ when using geometrically diminishing stepsizes.
\begin{lem} \label{lem: consensus-error-sharpness}
Let the stepsizes in DRSM be chosen as $\beta_k = \mu_0 \gamma^k, k\geq 0$, where $0 <\mu_0 \leq \min \left\{\frac{1-\rho_t}{L} \delta_1, \frac{\alpha \delta_1}{5L} \right\}$ and $\rho_t^\delta \leq \gamma <1, \delta\in(0,1)$. If Assumption \ref{assum: Connected graph} holds and $\vx_0 \in \calN$, then we have $\|\vx_k - \bar{\vx}_k \|_F = \calO(\beta_k)$ .
\end{lem}
\begin{proof}
		The condition $\mu_0 \leq \min \left\{\frac{1-\rho_t}{L} \delta_1, \frac{\alpha \delta_1}{5L} \right\}$ implies $\beta_k \leq \min \left\{\frac{1-\rho_t}{L} \delta_1, \frac{\alpha \delta_1}{5L} \right\}$ for all $k \geq 0$. Hence, the condition on $\beta_k$ in Lemma \ref{thm: consensusbound} is satisfied. The inequality \eqref{eq: consensusboundgeneralbeta} yields
\begin{align}
	\| \vx_{k+1} - \bar{\vx}_{k+1}\|_F
	& \leq \rho_t^{k+1} \|\vx_0 - \bar{\vx}_0 \|_F + \sqrt{n}L \sum_{l=0}^k \rho_t^{k-l}\beta_l \nonumber\\
	& \leq \|\vx_0 - \bar{\vx}_0 \|_F  \gamma^{k+1} + \frac{\mu_0\sqrt{n}L}{\gamma} \left(\sum_{l=0}^k \left(\frac{\rho_t}{\gamma}\right)^{k-l} \right) \gamma^{k+1} \nonumber \\
	& \leq \left( \|\vx_0 - \bar{\vx}_0 \|_F  + \frac{\mu_0 \sqrt{n}L}{\gamma(1-\gamma^{\frac{1}{\delta}-1})} \right) \gamma^{k+1}, \label{eq: consensuserrorsharp}
\end{align}
where the second inequality is from $\rho_t \leq \gamma$.
\end{proof}

The following two assumptions are required to prove the local linear convergence of DRSM under the sharpness condition.
\begin{assum}[Isolated optimal solutions] \label{assum: isolated}
There exists a weak sharp minimum $x^* \in \calM$ of problem \eqref{eq: dwcoptstiefel0} that is isolated.
\end{assum}
\begin{assum} \label{assum: mu0}
Let $\vx_0^\top = (x_{1,0}^\top,\dots,x_{n,0}^\top)$ be the initial point of DRSM. Let $\Gamma \geq 3$ be a constant. Define
{\small	\begin{align}
		&e_0\coloneqq \min \left\{  \max \left\{ \frac{\kappa}{(\rho+L) \Gamma}, \sqrt{\sum_{i=1}^n \frac{\| x_{i,0} - x^* \|_F^2}{n} }\right\}, \frac{B}{\Gamma}\right\}, \label{eq: e0-def}\\
		&  a \coloneqq 2(L+\kappa +\alpha LL_t)L,\quad b \coloneqq (4\alpha \sqrt{r} + 2\alpha^2r)L_tL^2, \label{eq: ab-def} \\
		& q \coloneqq \frac{2\kappa e_0}{\Gamma} - (\rho+L)e_0^2.
	\end{align}
}
Then, the constant $\mu_0 >0$ is strictly less than
{\small	\begin{align} \label{eq: assummu0}
		\min \left\{\frac{e_0 }{2\kappa - (\rho+L)e_0}, \frac{q}{L^2e_0^2 + \frac{4(a+b)}{(1-\rho_t)^2}}, \frac{(1-\rho_t)e_0}{2\sqrt{n}L} \right\}.
	\end{align}
}
\end{assum}
\begin{remark}
Although Assumption \ref{assum: isolated} seems a little restrictive, we emphasize that existing results in the Euclidean setting also rely on similar assumptions \citep[Definition \uppercase\expandafter{\romannumeral2}.2]{chen2021distributed}. Moreover, even with such an additional assumption, it is challenging to establish the local linear convergence of DRSM under the sharpness condition. We leave the question of whether Assumption \ref{assum: isolated} is needed in the proof of the said result as future work.
\end{remark}

With the above setup, we establish the local linear convergence result in the following theorem.
\begin{thm}[Local linear convergence under the sharpness condition] \label{thm: sharplinear} Suppose that the conditions in Lemma \ref{lem: consensus-error-sharpness}, Assumption \ref{assum: isolated}, and Assumption \ref{assum: mu0} hold. Suppose further that the initial point $\vx_0$ satisfies the following two conditions:
\begin{align}
	\sum_{i=1}^n \| x_{i,0} - x^* \|_F^2 & < \frac{n}{\Gamma^2} \min\left\{\left( \frac{2 \kappa}{\rho+L}\right)^2, B^2 \right\}, \label{eq: sharpnessthmcondition1}\\
	\| \vx_0 - \bar{\vx}_0\|_F & =0.\label{eq: sharpnessthmcondition3}
\end{align}
Then, there exists a sufficiently small constant $\delta>0$ such that for $\gamma = \rho_t^\delta$, we have 
\begin{align}
	& \sum_{i=1}^n \| x_{i,k} - x^*\|_F^2 \leq n \gamma^{2k}e_0^2, \label{eq: sharpconclu1}\\
	& \| x_{i,k} - x^*\|_F^2 \leq \Gamma^2 \gamma^{2k}e_0^2, \quad \forall i \in [n]\label{eq: sharpconclu2}
\end{align}
for any iterative sequence $\{x_{i,k}\}_{k\geq 0}$ generated by DRSM.
\end{thm}
\begin{proof}
		We prove it by induction. First, when $k=0$, we need to show 
\begin{align}
	\sum_{i=1}^n \| x_{i,0} - x^*\|_F^2 \leq ne_0^2, \text{ and }\| x_{i,0} - x^*\|_F^2 \leq \Gamma^2e_0^2 \label{eq: sharpnessthmpf0}.
\end{align}
Due to the facts of $x_{1,0} = \dots = x_{n,0}$ by \eqref{eq: sharpnessthmcondition3} and $\Gamma \geq 3$ in Assumption \ref{assum: mu0}, it suffices to show $\| x_{1,0} - x^*\|_F^2 \leq e_0^2$.
This can be done by considering the following two cases:
\begin{enumerate}[i)]
	\item If $\sqrt{\frac{1}{n} \sum_{i=1}^n \| x_{i,0} - x^*\|_F^2} \leq \frac{\kappa}{(\rho+L) \Gamma}$, which implies $\| x_{1,0} - x^* \|_F^2 \leq \frac{\kappa^2}{(\rho+L)^2 \Gamma^2 } $, then $e_0 = \min \left\{ \frac{\kappa}{(\rho+L) \Gamma}, \frac{B}{\Gamma} \right\}$. From condition \eqref{eq: sharpnessthmcondition1} we have $\| x_{1,0} - x^* \|_F^2 < \frac{B^2}{\Gamma^2}$. Thus, it follows that
	\begin{equation*}
		\| x_{1,0} - x^* \|_F^2 \leq \min \left\{ \frac{B^2}{\Gamma^2},  \frac{\kappa^2}{(\rho+L)^2 \Gamma^2 } \right\} = e_0^2.
	\end{equation*}
	\item If $\sqrt{\frac{1}{n} \sum_{i=1}^n \| x_{i,0} - x^*\|_F^2} > \frac{\kappa}{(\rho+L) \Gamma}$, then we obtain
	\[
	e_0 = \min \left\{ \sqrt{\frac{1}{n} \sum_{i=1}^n \| x_{i,0} - x^*\|_F^2}, \frac{B}{\Gamma} \right\} = \sqrt{\frac{1}{n} \sum_{i=1}^n \| x_{i,0} - x^*\|_F^2} =  \| x_{1,0} - x^*\|_F^2,
	\]
	where the second equality is due to \eqref{eq: sharpnessthmcondition1}.
\end{enumerate}

Assume that the inequalities \eqref{eq: sharpconclu1} and \eqref{eq: sharpconclu2} hold for $k\geq 0$, it is required to prove that similar inequalities hold for $k+1$. According to the updating rule of \eqref{eq: algiteratedrsm}, we have 
\begin{align}
	& \quad \sum_{i=1}^n \| {x}_{i,k+1} - x^*\|_F^2 \nonumber \\
	& = \sum_{i=1}^n\| \calR_{x_{i,k}} \left(-\alpha \operatorname{grad} \varphi^t_i(\vx_k) - \beta_k \tilde{\nabla}_\calR f_i(x_{i,k}) \right) - x^* \|_F^2  \nonumber \\
	& \leq  \sum_{i=1}^n \| x_{i,k} - \alpha \operatorname{grad} \varphi^t_i(\vx_k) - \beta_k \tilde{\nabla}_\calR f_i(x_{i,k}) - x^*  \|_F^2 \nonumber \\
	& \leq \sum_{i=1}^n  \| x_{i,k} - \alpha \operatorname{grad} \varphi^t_i(\vx_k) - x^* \|_F^2 - \sum_{i=1}^n 2 \left\langle x_{i,k} - \alpha \operatorname{grad} \varphi^t_i(\vx_k) - x^*, \beta_k \tilde{\nabla}_\calR f_i(x_{i,k})\right\rangle +  n \beta_k^2 L^2, \label{eq: sharptheoryiter}
\end{align}
where the first inequality is due to properties of polar retraction in \citet[Lemma 2.3]{chen2021decentralized} and the last inequality comes from the fact $\| \tilde{\nabla}_\calR f_i(x_{i,k})\|_F \leq L$ in Section \ref{subsec: subdifferential}.
Regarding the first part in the right-hand side of \eqref{eq: sharptheoryiter}, it follows from \citet[(D.8)]{chen2021decentralized} that
\begin{align*}
	&\quad \sum_{i=1}^n  \| x_{i,k} - \alpha \operatorname{grad} \varphi^t_i(\vx_k) - x^* \|_F^2 \\
	& = \sum_{i=1}^n \|(1-\alpha)x_{i,k} + \alpha \sum_{j=1}^n W_{ij}^t x_{j,k} + \frac{\alpha}{2} x_{i,k} \sum_{j=1}^n W_{ij}^t(x_{i,k} - x_{j,k})^\top (x_{i,k} - x_{j,k}) -x^* \|_F^2
	\\
	& \leq \sum_{i=1}^n  \| (1-\alpha)x_{i,k} + \alpha \sum_{j=1}^n W_{ij}^t x_{j,k} - x^*\|_F^2 + \sum_{i=1}^n \| \frac{\alpha}{2} \sum_{j=1}^n W_{ij}^t(x_{i,k} - x_{j,k})^\top (x_{i,k} - x_{j,k})\|_F^2 \\
	&\quad +\sum_{i=1}^n  \left \langle \alpha x_{i,k} \sum_{j=1}^n W_{ij}^t(x_{i,k} - x_{j,k})^\top (x_{i,k} - x_{j,k}),  (1-\alpha)x_{i,k} + \alpha \sum_{j=1}^n W_{ij}^t x_{j,k} - x^* \right\rangle\\
	&\leq \sum_{i=1}^n \| x_{i,k}-x^*\|_F^2 +  2 \sqrt{r} \sum_{i=1}^n \alpha\|  \sum_{j=1}^n W_{ij}^t(x_{i,k} - x_{j,k})^\top (x_{i,k} - x_{j,k})\|_F  \\
	& \quad + \sum_{i=1}^n \| \frac{\alpha}{2} \sum_{j=1}^n W_{ij}^t(x_{i,k} - x_{j,k})^\top (x_{i,k} - x_{j,k})\|_F^2 \\
	& \leq \sum_{i=1}^n \| x_{i,k}-x^*\|_F^2 +  2 \sqrt{r} \sum_{i=1}^n \alpha\|  \sum_{j=1}^n W_{ij}^t(x_{i,k} - x_{j,k})^\top (x_{i,k} - x_{j,k})\|_F \\
	& \quad + \sum_{i=1}^n \alpha^2 r \| \sum_{j=1}^n W_{ij}^t(x_{i,k} - x_{j,k})^\top (x_{i,k} - x_{j,k})\|_F \\
	& \leq \sum_{i=1}^n \| x_{i,k}-x^*\|_F^2  + 4\alpha \sqrt{r}L_t \|\vx_k - \bar{\vx}_k\|_F^2 + 2\alpha^2r L_t \|\vx_k - \bar{\vx}_k\|_F^2 .
\end{align*}
where the second inequality uses \citet[(\uppercase\expandafter{\romannumeral2}.3)]{chen2021distributed}, the third inequality is due to the fact that
\begin{equation*}
	\| \sum_{j=1}^n W_{ij}^t(x_{i,k} - x_{j,k})^\top (x_{i,k} - x_{j,k})\|_F \leq  \sum_{j=1}^n W_{ij}^t \| x_{i,k} - x_{j,k} \|_F^2 \leq 4r,
\end{equation*}
and the last inequality follows from \citet[Proof of Lemma 10]{chen2021local} that 
\begin{equation}
	\sum_{i=1}^n\|  \sum_{j=1}^n W_{ij}^t(x_{i,k} - x_{j,k})^\top (x_{i,k} - x_{j,k})\|_F \leq 2L_t \| \vx	_k - \bar{\vx}_k\|_F^2.
\end{equation}

Next, let us bound the second part in the right-hand side of \eqref{eq: sharptheoryiter}, by the Riemannian subgradient inequality \eqref{eq: riessubgineq} and the sharpness condition \eqref{eq: sharpdef1}, we obtain 
\begin{align*}
	& \sum_{i=1}^n 2\langle -x_{i,k} + \alpha \operatorname{grad} \varphi^t_i(\vx_k) + x^*, \beta_k \tilde{\nabla}_\calR f_i(x_{i,k})\rangle \\
	=&  \sum_{i=1}^n 2 \beta_k \langle   x^* - x_{i,k},  \tilde{\nabla}_\calR f_i(x_{i,k}) \rangle +  \sum_{i=1}^n 2\alpha \beta_k\langle  \operatorname{grad}\varphi^t_i(\vx_k) ,  \tilde{\nabla}_\calR f_i(x_{i,k}) \rangle\\
	\leq &   \sum_{i=1}^n \left(2 \beta_k \left(f_i(x^*) - f_i(\bar{x}_k) + f_i(\bar{x}_k)- f_i(x_{i,k})\right)+ \beta_k(\rho +L)\|x_{i,k} - x^* \|_F^2 \right) + 2\alpha L \beta_k \sqrt{n}L_t \| \vx_k - \bar{\vx}_k\|_F, \\
	\leq & \sum_{i=1}^n \left( 2L\beta_k \|\bar{x}_k - x_{i,k} \|_F + \beta_k(\rho+L)\| x_{i,k} - x^*\|_F^2 \right) - 2n\kappa \beta_k \| \bar{x}_k - x^* \|_F + 2\alpha L \beta_k \sqrt{n}L_t \| \vx_k - \bar{\vx}_k\|_F,
\end{align*}
where the first inequality uses \citet[(C.9)]{chen2021decentralized} and the last inequality is due to the fact in Lemma \ref{lem: riessubgineq} that $f_i$ is $L$-Lipschitz continuous and the fact $\| \bar{x}_k - x^*\|_F \leq 2e_0 \leq \frac{2B}{\Gamma} <B $ by \eqref{eq: dist-barx-xopt-usedinproof} and \eqref{eq: e0-def}.

Plugging the above bounds of two parts into \eqref{eq: sharptheoryiter} gives
\begin{align}
	& \sum_{i=1}^n \| {x}_{i,k+1} - x^*\|_F^2 \nonumber \\
	\leq &  \sum_{i=1}^n \left((1+ \beta_k(\rho+L))\| x_{i,k} - x^*\|_F^2 + 2L\beta_k \| x_{i,k} -\bar{x}_k  \|_F  \right) - 2n\kappa \beta_k \| \bar{x}_k - x^* \|_F \nonumber \\
	& +  \left(4\alpha \sqrt{r} + 2\alpha^2 r\right)L_t \|\vx_k - \bar{\vx}_k\|_F^2+ 2\alpha L \beta_k \sqrt{n}L_t \| \vx_k - \bar{\vx}_k\|_F + n \beta_k^2 L^2 \nonumber \\
	\leq & \sum_{i=1}^n \left((1+ \beta_k(\rho+L))\| x_{i,k} - x^*\|_F^2 - 2\kappa \beta_k \| x_{i,k} -x^*  \|_F   + 2(L+\kappa) \beta_k \| x_{i,k} - \bar{x}_k \|_F \right) \label{eq: sharpinductionconpre} \\
	& + \left(4\alpha \sqrt{r} + 2\alpha^2 r\right)L_t \|\vx_k - \bar{\vx}_k\|_F^2+ 2\alpha L \beta_k \sqrt{n}L_t \| \vx_k - \bar{\vx}_k\|_F + n \beta_k^2 L^2, \nonumber
\end{align}
where the second inequality is due to $n\|\bar{x}_k - x^*\|_F \geq \sum_{i=1}^n (\| x_{i,k} - x^*\|_F - \|x_{i,k} - \bar{x}_k \|_F)$. Now, we derive an upper bound for the term 
\[
\sum_{i=1}^n \left((1+ \mu_0(\rho+L))\| x_{i,k} - x^*\|_F^2 - 2\kappa \beta_k \| x_{i,k} -x^*  \|_F \right).
\]
From \eqref{eq: e0-def} and \eqref{eq: sharpnessthmcondition1}, it holds 
$2\kappa - (\rho+L)e_0 > 0$. 
By
\begin{align}
	\mu_0 \leq  \frac{e_0}{2\kappa - (\rho+L)e_0},
\end{align}
we have $\frac{2\kappa \beta_k}{1+\mu_0(\rho+L)} = \frac{2\kappa \mu_0 \gamma^k}{1+\mu_0(\rho+L)} \leq e_0 \gamma^k$. By invoking \citet[Lemma A.1]{chen2021distributed} (with $a = e_0 \gamma^k$, $b= \frac{\kappa \beta_k}{ 1+ \mu_0(\rho + L) } $, and $c = \Gamma$ in the lemma), we obtain
\begin{align}
	& \left(1+ \mu_0(\rho + L) \right)  \sum_{i=1}^n \left(  \| x_{i,k} - x^*\|_F^2 - \frac{2 \kappa \beta_k}{\left( 1+ \mu_0(\rho + L) \right) }  \|x_{i,k} - x^* \|_F \right) \nonumber \\
	\leq & \left(1+ \mu_0(\rho + L) \right)ne_0^2\gamma^{2k} - \frac{2n}{\Gamma}\kappa \beta_k e_0 \gamma^k.  \label{eq: sharp optdistquad upperbound}
\end{align}

Combining \eqref{eq: sharp optdistquad upperbound}, \eqref{eq: sharpinductionconpre} with the consensus error in \eqref{eq: consensuserrorsharp} yields
\begin{align}
	&\sum_{i=1}^n \| {x}_{i,k+1} - x^*\|_F^2 \nonumber\\
	\leq & \left(1+ \mu_0(\rho + L) \right)ne_0^2\gamma^{2k} - \frac{2n}{\Gamma}\kappa \beta_k e_0 \gamma^k + 2\sqrt{n}(L+\kappa)\beta_k  \| \vx_k - \bar{\vx}_k\|_F \nonumber\\
	& + \left(4\alpha \sqrt{r} + 2\alpha^2r \right)L_t \|\vx_k - \bar{\vx}_k\|_F^2 + 2\alpha L \beta_k \sqrt{n}L_t \| \vx_k - \bar{\vx}_k\|_F + n \beta_k^2 L^2 \nonumber \\
	\leq & \left(1+ \mu_0(\rho + L) \right)ne_0^2\gamma^{2k} - \frac{2n}{\Gamma}\kappa \mu_0 e_0 \gamma^{2k}
	+ \left( 2\sqrt{n}(L+\kappa)+2\alpha\sqrt{n}LL_t \right)\mu_0\gamma^k  \left(  \frac{\mu_0 \sqrt{n}L}{\gamma(1-\gamma^{\frac{1}{\delta}-1})} \right) \gamma^{k} \nonumber  \\
	& +  \left(4\alpha \sqrt{r} + 2\alpha^2r \right)L_t \left(  \frac{\mu_0 \sqrt{n}L}{\gamma(1-\gamma^{\frac{1}{\delta}-1})} \right)^2 \gamma^{2k} + n \mu_0^2 L^2 \gamma^{2k}. \label{eq: sharpquaduppbound}
\end{align} 

Grouping the terms related to $\mu_0$ in \eqref{eq: sharpquaduppbound} gives
\begin{align*}
	\sum_{i=1}^n \| {x}_{i,k+1} - x^*\|_F^2 
	\leq ne_0^2 \gamma^{2k} \left(1  -\frac{q}{e_0^2} \mu_0  + \frac{   \frac{2(L+\kappa +\alpha LL_t)L}{\gamma(1-\gamma^{1/\delta-1})} + \frac{(4\alpha \sqrt{r} + 2 \alpha^2r)L_tL^2}{\gamma^2(1-\gamma^{1/\delta-1})^2}  + L^2  }{e_0^2} \mu_0^2\right),
\end{align*}
where $q= \frac{2\kappa e_0}{\Gamma} - (\rho+L)e_0^2$.
By the fact that $e_0 < \frac{2\kappa}{\Gamma(\rho+L)}$, $q >0$ can be guaranteed.
Note that $\gamma \in (0,1)$ and if the following condition holds
\begin{equation}
	\gamma^2 \geq
	1  -\frac{q}{e_0^2} \mu_0+  \frac{   \frac{2(L+\kappa +\alpha LL_t)L}{\gamma(1-\gamma^{1/\delta-1})} + \frac{(4\alpha \sqrt{r} + \alpha^2/2)L_tL^2}{\gamma^2(1-\gamma^{1/\delta-1})^2}  + L^2  }{e_0^2} \mu_0^2,
	\label{eq: sharpness last condition proof}
\end{equation}
we readily obtain
\begin{align*}
	\sum_{i=1}^n \| {x}_{i,k+1} - x^*\|_F^2 \leq n e_0^2 \gamma^{2(k+1)}.
\end{align*}
Proof of \eqref{eq: sharpness last condition proof}: Since $0 < \rho_t^{\delta} - \rho_t < 1$, it suffices to prove the following inequality,
\begin{equation}
	\label{eq: sharpness last condition proof1}
	\rho_t^{2\delta} \geq 1  -\frac{q}{e_0^2} \mu_0 + \frac{L^2}{e_0^2}\mu_0^2 + \frac{a+b}{(\rho_t^\delta - \rho_t)^2e_0^2} \mu_0^2,
\end{equation}
where $a$ and $b$ are defined in \eqref{eq: ab-def}. 
From the condition $\mu_0 < \frac{q}{L^2e_0^2 + \frac{4(a+b)}{(1-\rho_t)^2}}$ given in \eqref{eq: assummu0}, we obtain there exists a sufficiently small $\delta \in (0,1)$ such that 
\begin{equation} \label{eq: pfrhot-delta-1/2}
	\rho_t^\delta - \rho_t \geq \frac{1-\rho_t}{2}
\end{equation}
and
\begin{align*}
	1  -\frac{q}{e_0^2} \mu_0 + \frac{L^2}{e_0^2}\mu_0^2 + \frac{a+b}{(\rho_t^\delta - \rho_t)^2e_0^2} \mu_0^2
	\leq  1  -\frac{q}{e_0^2} \mu_0 + \left(L^2+\frac{4(a+b)}{(1-\rho_t)^2}\right)\frac{\mu_0^2}{e_0^2}
	<1.
\end{align*}
Taking a sufficiently small $\delta \in (0,1)$ gives \eqref{eq: sharpness last condition proof1}.

Last, we verify \eqref{eq: sharpconclu2} for $k+1$. According to $\| \bar{x}_{k+1} - x^* \|_F^2 \leq \frac{4}{n} \sum_{i=1}^n \| x_{i,k+1} - x^* \|_F^2 \leq 4e_0^2 \gamma^{2k+2}$ (by Lemma \ref{lem: liplikepro}), we have
\begin{align*}
	\| x_{i,k+1} - x^* \|_F
	\leq & \| x_{i,k+1} - \bar{x}_{k+1} \|_F + \| \bar{x}_{k+1} - x^* \|_F
	\leq \left( \frac{\mu_0 \sqrt{n}L}{\gamma(1-\gamma^{\frac{1}{\delta}-1})} \right) \gamma^{k+1} + 2e_0 \gamma^{k+1}\\
	\leq & \left( \frac{\mu_0 \sqrt{n}L}{\rho_t^\delta - \rho_t} + 2e_0 \right) \gamma^{k+1}
	\leq \left( \frac{2\mu_0 \sqrt{n}L}{1 - \rho_t} + 2e_0 \right) \gamma^{k+1} 
	\leq \Gamma e_0 \gamma^{k+1},
\end{align*}
where the second inequality comes from \eqref{eq: consensuserrorsharp}, the third inequality is by $\gamma = \rho_t^\delta$, the fourth inequality is by \eqref{eq: pfrhot-delta-1/2}, and the last inequality is due to \eqref{eq: assummu0}.
\end{proof}

Some comments on Theorem \ref{thm: sharplinear} are in order.
\begin{enumerate}[i)]
\item The condition \eqref{eq: sharpnessthmcondition1} requires that the initial points $x_{i,0}, i\in [n]$ should be all close to $x^*$. One can simply initialize all the agents with the same value, i.e., $x_{1,0} = x_{2,0}=\dots=x_{n,0}$, to satisfy the condition \eqref{eq: sharpnessthmcondition3}. 
\item An immediate conclusion is that 
\begin{align}
	\|\bar{x}_k - x^* \|_F^2 \leq 4\|\hat{x}_k -x^*\|_F^2 \leq \frac{4}{n} \sum_{i=1}^n\|x_{i,k}-x^*\|_F^2 \leq 4 \gamma^{2k}e_0^2, \label{eq: dist-barx-xopt-usedinproof}
\end{align}
where the first inequality comes from Lemma \ref{lem: liplikepro} and the last inequality is from \eqref{eq: sharpconclu1}.
\item The key difference between the Euclidean setting and Stiefel manifold settings when proving the local linear convergence under the sharpness condition lies in the following distinct inequalities. Since the consensus step performs linear operations in the Euclidean setting, one can show that
\begin{equation}
	\sum_{i=1}^n \| \sum_{j=1}^n W_{ij}^t x_{j,k} - x^* \|_F^2 \leq \sum_{i=1}^{n} \left\| x_{i,k} - x^* \right\|^2_F
\end{equation} 
without too much difficulty.
However, in the Stiefel manifold setting, as the consensus step involves not just linear operations, we need to carefully provide an upper bound for $\sum_{i=1}^n  \| x_{i,k} - \alpha \operatorname{grad} \varphi^t_i(\vx_k) - x^* \|_F^2$.
\end{enumerate}

\section{Numerical Experiments}
We conduct numerical experiments on the decentralized dual principal component pursuit (DPCP) problem as well as the decentralized orthogonal dictionary learning (ODL) problem to compare our DRSM algorithm with its centralized counterpart~(CRSM)~\citep{li2021weakly}.

Throughout the experiments, for the network topology, we consider three different choices: A complete graph, a ring graph, and an Erd\"os-R\'enyi (ER) random graph where each possible edge is generated independently with probability $0.3$.

\subsection{Dual principal component pursuit (DPCP)}
In the DPCP problem, one is given some measurements $\tilde{Y}= [Y\ O]\Gamma \in \R^{d \times m}$, where the columns of $Y \in \R^{d \times m_1}$ form inlier points spanning a $(d-r)$-dimensional subspace $\calS$, the columns of $O \in \R^{d \times m_2}$ form outlier points with no linear structure, and $\Gamma \in \R^{m \times m}$ is an unknown permutation. To recover the subspace $\calS$ (or $\calS^\perp$), one aims to solve
\begin{equation}\label{eq: DPCPpro}
\begin{aligned}
	& \min_{X \in \R^{nd \times r}} \;f(X) \coloneqq  \frac{1}{n} \sum_{i=1}^n \left( \frac{1}{N} \sum_{j=1}^N \left\| (\tilde{y}_{i,j})^\top X_i  \right\|_2 \right) \\
	& \text { s.t. } \ X_{1}= X_{2}=\dots=X_{n},\, X_{i} \in {\rm St}(d,r), \, \forall i \in [n],
\end{aligned}
\end{equation}
where $X^\top \coloneqq (X_1^\top, X_2^\top,\dots,X_n^\top)$, $m = n \times N$, and $\tilde{y}_{i,j} \in \R^{d}$ is the $j$-th column vector of the data in the $i$-th local node. 

We generate the measurements $\tilde{Y}$ following the work of \citet{li2021weakly} with $d=100$ and $r=10$. Specifically, the randomly generated subspace $\calS$ has dimension $(d-r)$ in a $d$-dimensional ambient space. Then, we generate $m_1=1500$ inliers uniformly at random from the unit sphere in $\calS$ and $m_2=3500$ outliers uniformly at random from the unit sphere in $\R^d$. After that, we randomly allocate those $m$ column vectors to $n=10$ local nodes such that each node has $N=500$ column vectors. The initialization is set to satisfy $X_1 = \dots = X_n$ and we randomly generate $X_1$ on ${\rm St}(d,r)$.

The DPCP problem is weakly convex and possesses the sharpness property with high probability under suitable conditions \citep{li2021weakly}. In our experiments, suppose that the underlying subspace $\calS^\perp$ is the column space of a matrix $X_{\rm true} \in {\rm St}(d,r)$. We measure the performance by the distance between the IMA in the $k$-th iteration and the low-dimensional subspace $\calS^\perp$, that is
\[
{\rm dist}(\bar{X}_k, \calS^\perp) = \min_{Q\in O(r)} \| \bar{X}_k Q - X_{\rm true} \|_F,
\]
where $O(r)$ represents the set of $r \times r$ orthogonal matrices.

First, we solve the DPCP problem by both DRSM and CRSM with diminishing stepsizes. In each epoch $k$, the stepsize for both algorithms is set to be $\beta_k = 0.05/\sqrt{k}$. Moreover, we simply set $t=1$ for the multistep consensus in our proposed DRSM. We demonstrate the sublinear convergence of DRSM and CRSM in Figure~\ref{fig:dpcp experiments}(a). We show that DRSM converges faster on the complete graph than on the ring and ER graphs. Besides, it is interesting to observe that the performance of DRSM on the complete graph is even better than that of CRSM.

Then, we present the linear convergence rate of DRSM with exponentially decaying stepsizes. In each epoch $k$, we set the stepsize for DRSM as $\beta_k = 0.05\times 0.98^{k}$, and for CRSM as $\beta_k = 0.05\times 0.9^{k}$. Again, we set $t=1$ for the multistep consensus in our DRSM. The convergence results are shown in Figure~\ref{fig:dpcp experiments}(b). From Figure~\ref{fig:dpcp experiments}(b), our proposed DRSM converges linearly in all three graphs, which is in line with our theoretical analysis. In Figure~\ref{fig:dpcp experiments}(c), we show the convergence performance of DRSM with geometrically diminishing stepsizes and varying $t$. It can be observed that the convergence behaviors of DRSM with different $t$ are similar.

\begin{figure*}[t]
\centering
\subfigure[Polynomial diminishing stepsizes]{\includegraphics[width=0.32\linewidth]{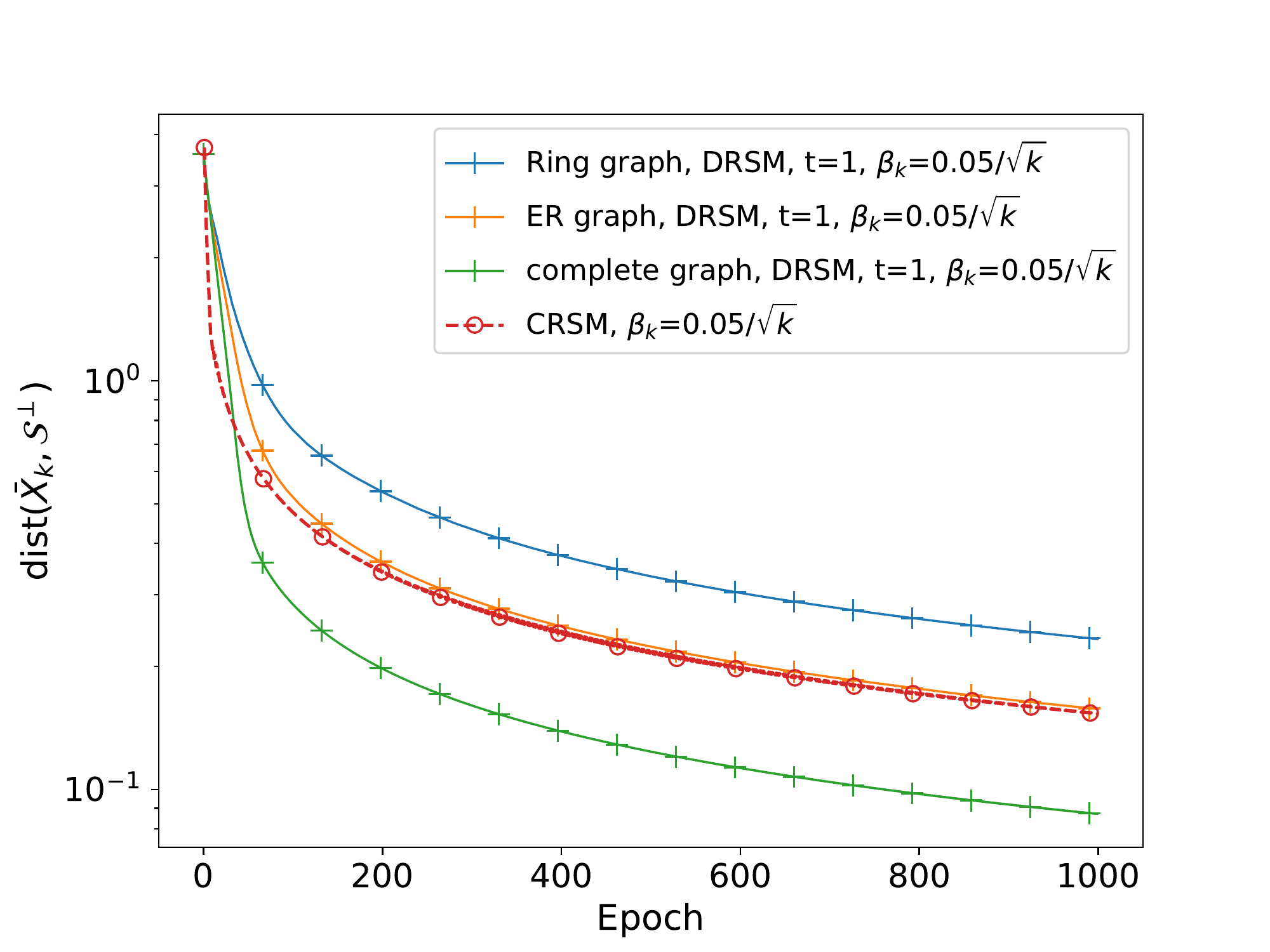}}
\subfigure[Geometrically diminishing stepsizes]{\includegraphics[width=0.32\linewidth]{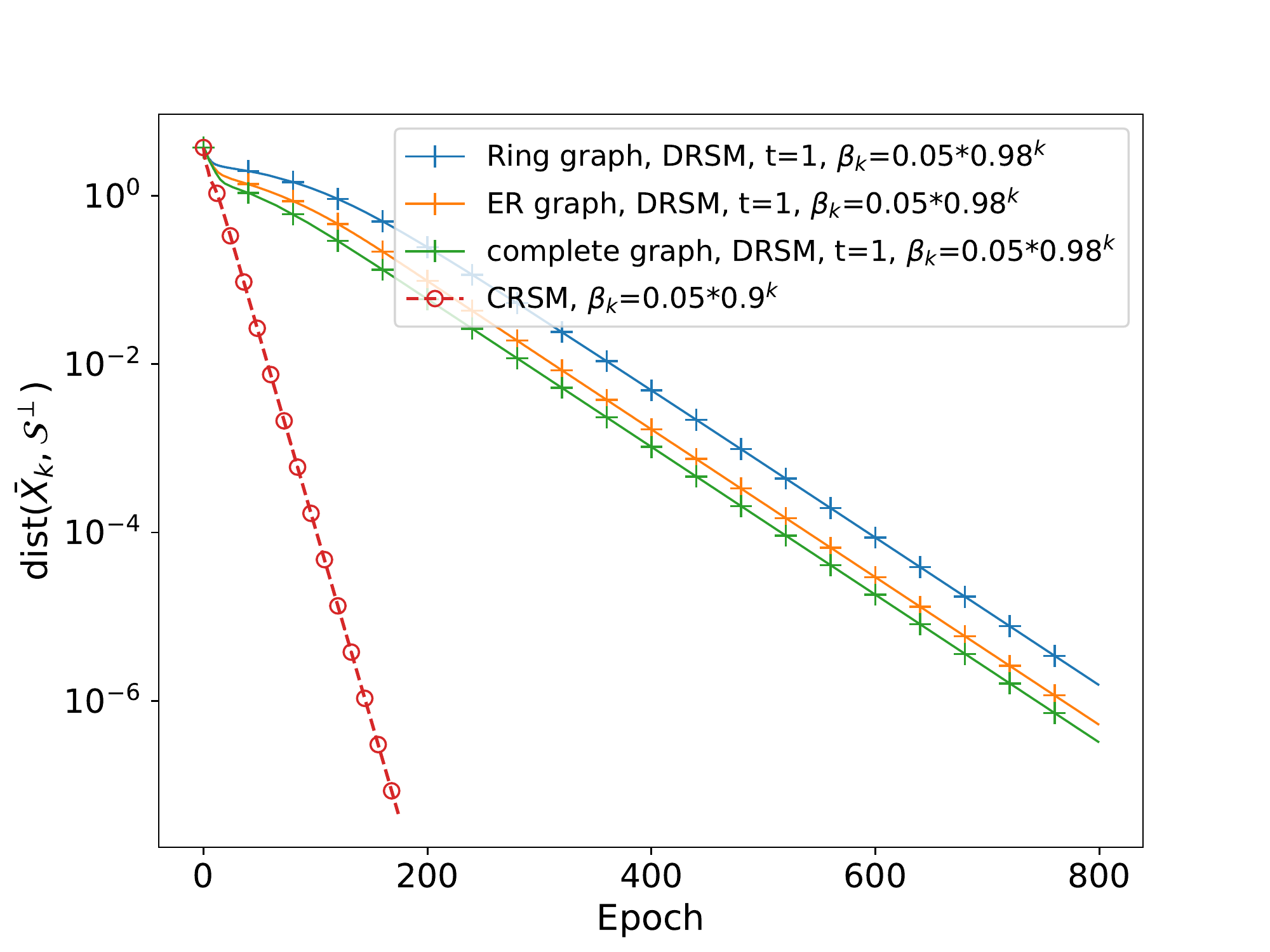}}
\subfigure[Geometrically diminishing stepsizes and different $t$]{\includegraphics[width=0.32\linewidth]{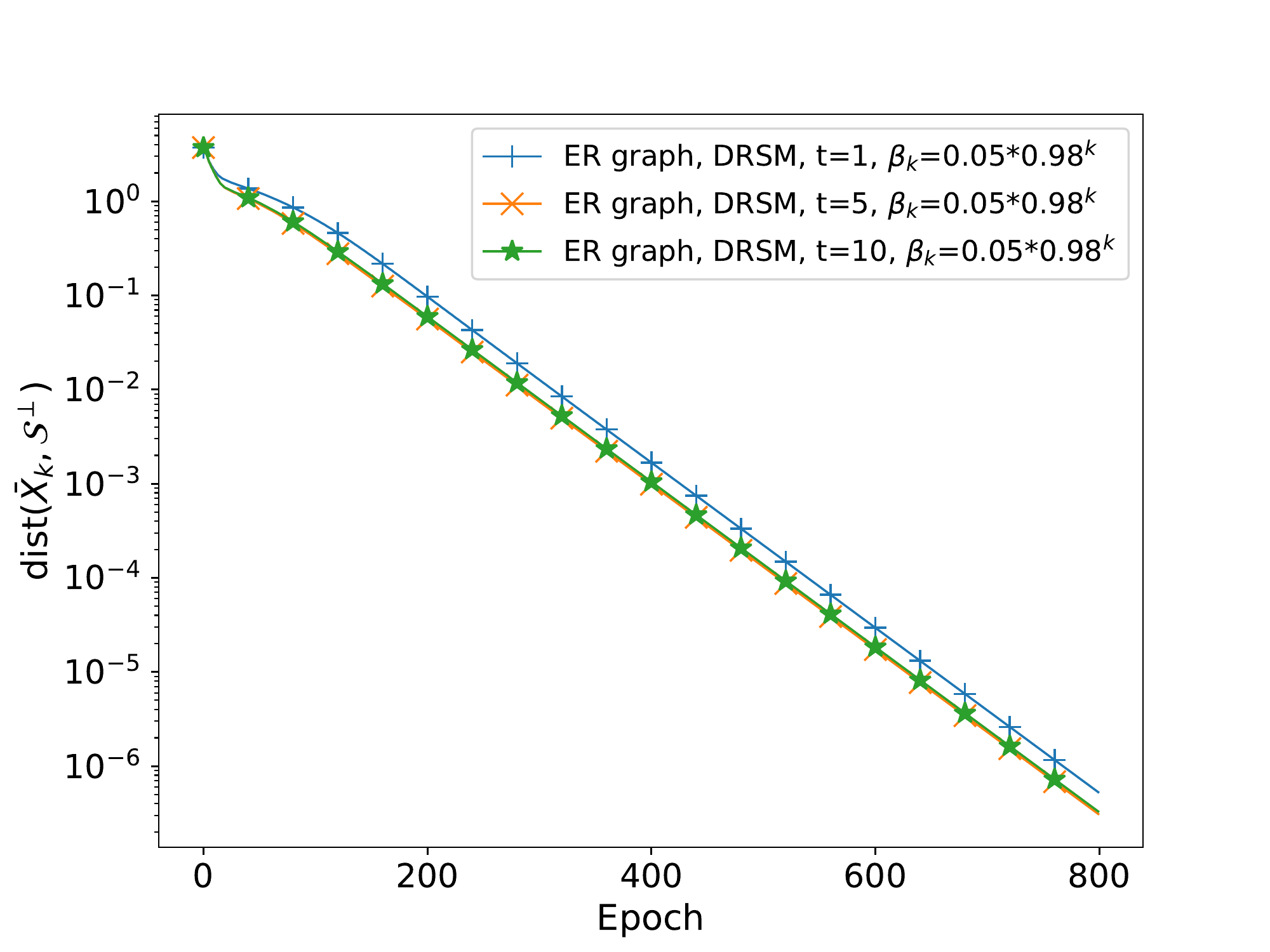}}
\vskip -1.0em
\caption{Convergence performance of Riemannian subgradient-type methods for the DPCP formulation.}
\label{fig:dpcp experiments}
\end{figure*}

\begin{figure*}[t]
\centering
\subfigure[Polynomial diminishing stepsizes]{\includegraphics[width=0.32\linewidth]{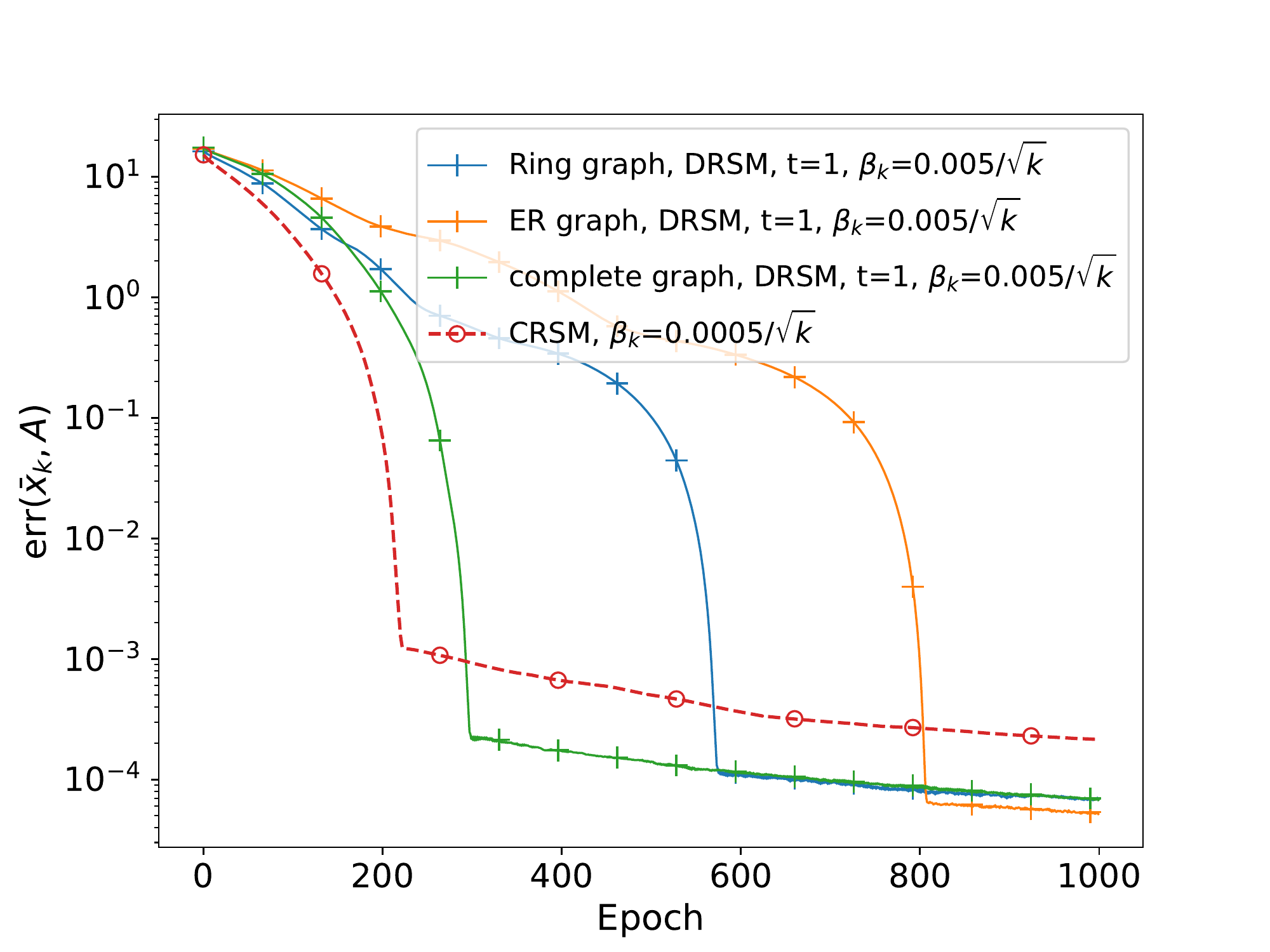}}
\subfigure[Geometrically diminishing stepsizes]{\includegraphics[width=0.32\linewidth]{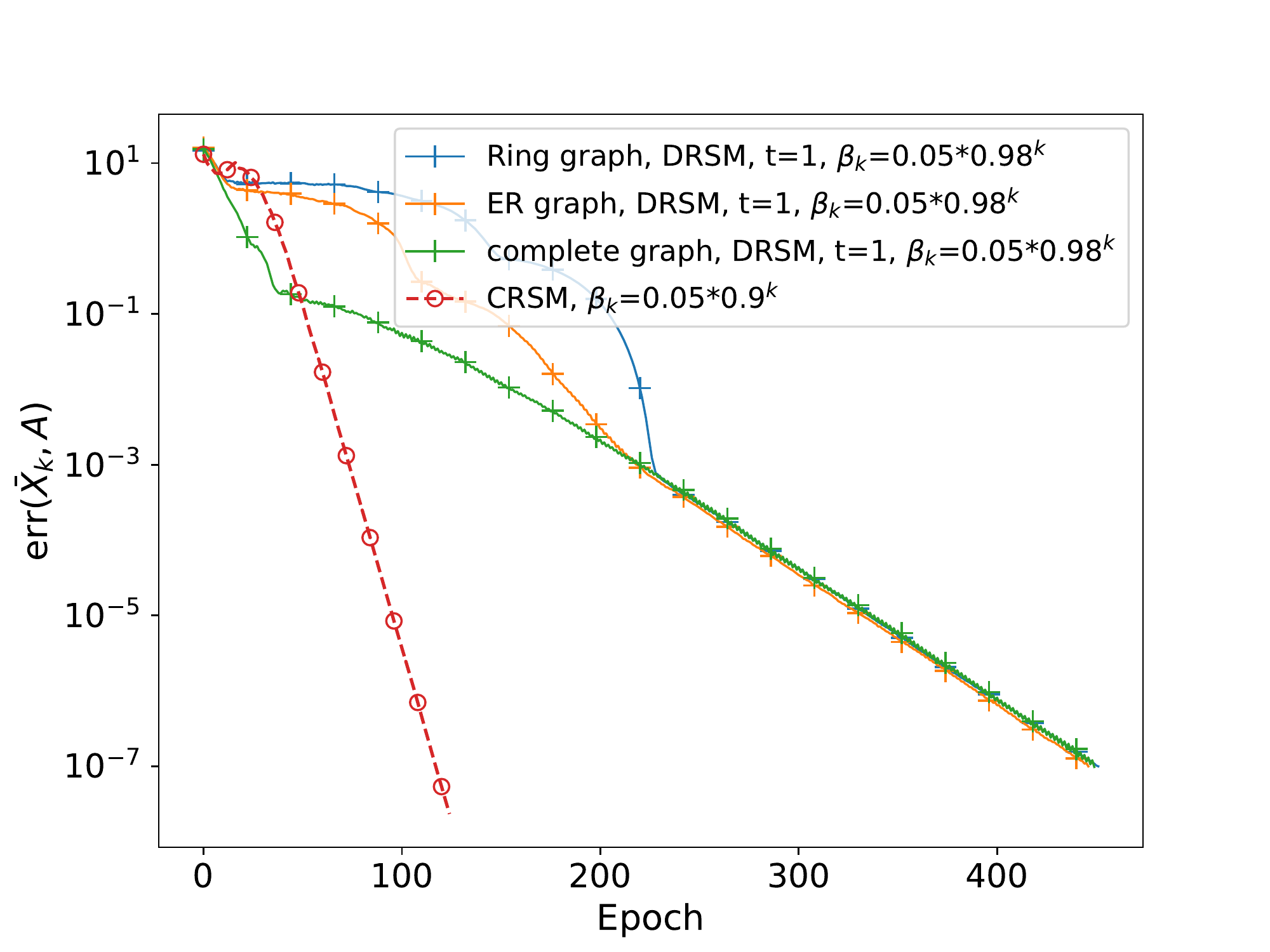}}
\subfigure[Geometrically diminishing stepsizes and different $t$]{\includegraphics[width=0.32\linewidth]{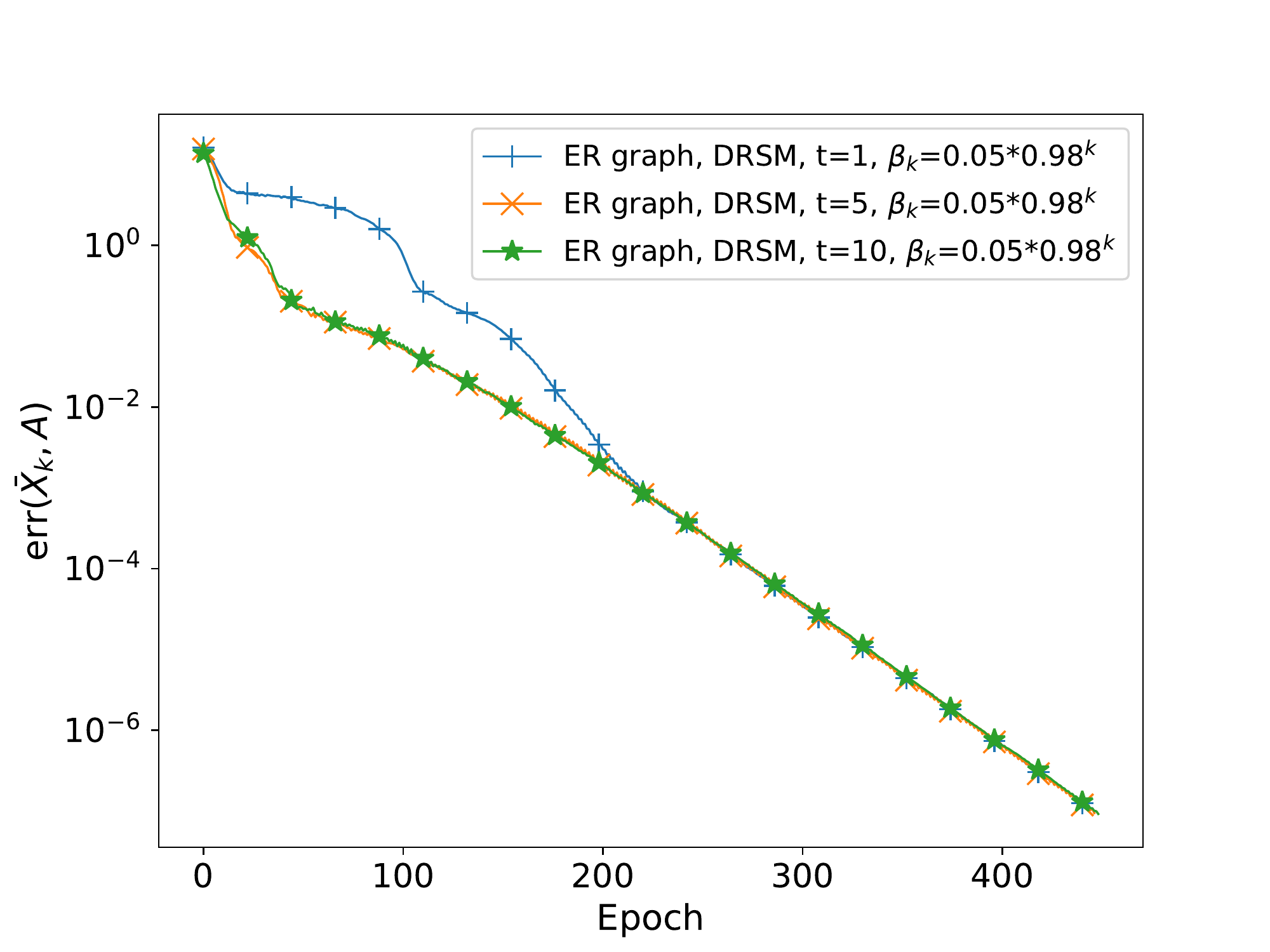}}
\vskip -1.0em
\caption{Convergence performance of Riemannian subgradient-type methods for the ODL formulation.}
\label{fig:odl experiments}
\end{figure*}

\subsection{Orthogonal dictionary learning (ODL)} 
For the ODL problem, the goal is to obtain a suitable compact representation of the observed data $Y\in \R^{d \times m}$. Assuming that the observation $Y$ can be approximated by $Y \approx AS$, where $A \in {\rm St}(d,d)$ represents the underlying orthogonal dictionary to be estimated and each column of $S\in \R^{d \times m}$ is sparse, we try to recover the entire dictionary $A$ by considering the formulation
\begin{equation}\label{eq: ODLpro}
\begin{aligned}
	& \min_{X \in \R^{nd \times d}} \;f(X) \coloneqq  \frac{1}{n} \sum_{i=1}^n \left( \frac{1}{N} \sum_{j=1}^N  \left\| (y_{i,j})^\top X_i  \right\|_1 \right) \\
	& \text { s.t. } \ X_{1}= X_{2}=\dots=X_{n},\, X_{i} \in {\rm St}(d,d), \, \forall i \in [n],
\end{aligned}
\end{equation}
where $X^\top = (X_1^\top, X_2^\top,\dots,X_n^\top)$, $m = n \times N$, and ${y_{i,j}} \in \R^{d}$ is the $j$-th column vector of the data in the $i$-th local node. 

We generate the underlying orthogonal dictionary $A \in {\rm St}(d,d)$ with $d= 30$ randomly and observe $m =1650 \approx 10 \times d^{1.5}$ instances by $Y = AS$, where each element of the sparse matrix $S$ is generated by the Bernoulli-Gaussian distribution with parameter $0.3$. Then, we randomly allocate these $m$ columns of $Y$ to $n=10$ local nodes with $N=165$ column vectors on each node. We also use random Gaussian initialization to generate $X_1 \in {\rm St}(d,d)$ and set $X_1 = \dots = X_n$. The performance measure is defined as the error between $\bar{X}$ and $A$; i.e., ${\rm err}(\bar{X},A) = \sum_{i=1}^d |\max_{1\leq j \leq d} |[\bar{X}_i^\top A]_j|-1 |$.

We first utilize the DRSM and CRSM to solve the ODL problem with diminishing stepsizes. Specifically, at each epoch $k$, we set the step size as $\beta_{k}=0.005/\sqrt{k}$ and $0.0005/\sqrt{k}$ for DRSM and CRSM, respectively. Figure~\ref{fig:odl experiments}(a) shows the sublinear convergence of our proposed DRSM in all three graphs. Similarly, we can observe that DRSM algorithm on the complete graph converges faster than the ring and ER graphs.

Then, Figure~\ref{fig:odl experiments}(b) shows the linear convergence of DRSM and CRSM when exponentially diminishing stepsizes of the form $\beta_k=\mu_0\gamma^k$ are used. Here, $\mu_0=0.05, \gamma=0.98$ are chosen for DRSM, and $\mu_0=0.05, \gamma=0.9$ are chosen for CRSM. We also show the performance of DRSM with geometrically diminishing stepsizes and varying $t$ in Figure~\ref{fig:odl experiments}(c). 
We can observe that DRSM with $t=5$ or $10$ converges faster in the initial process than DRSM with $t=1$ and they behave similarly later.

\section{Concluding Remarks}\label{sec:con}
We proposed the decentralized Riemannian subgradient method (DRSM) for solving decentralized weakly convex (possibly non-smooth) optimization problems over the Stiefel manifold and established its global convergence and iteration complexity. Besides, the method enjoys a local linear convergence rate if the problem at hand exhibits the sharpness property. Future directions include exploring practical optimization problems over other embedded manifolds and provably alleviating the communication burden since multiple rounds of communications are required per iteration in DRSM.

\bibliography{decen-weakcov-stiefel.bib}

\begin{thebibliography}{49}
\providecommand{\natexlab}[1]{#1}
\providecommand{\url}[1]{\texttt{#1}}
\expandafter\ifx\csname urlstyle\endcsname\relax
  \providecommand{\doi}[1]{doi: #1}\else
  \providecommand{\doi}{doi: \begingroup \urlstyle{rm}\Url}\fi

\bibitem[Absil et~al.(2009)Absil, Mahony, and Sepulchre]{absil2009optimization}
P.-A. Absil, R.~Mahony, and R.~Sepulchre.
\newblock \emph{Optimization Algorithms on Matrix Manifolds}.
\newblock Princeton University Press, 2009.

\bibitem[Balashov and Tremba(2020)]{balashov2020error}
M.~Balashov and A.~Tremba.
\newblock Error bound conditions and convergence of optimization methods on
  smooth and proximally smooth manifolds.
\newblock \emph{Optimization}, pages 1--25, 2020.

\bibitem[Boumal et~al.(2019)Boumal, Absil, and Cartis]{boumal2019global}
N.~Boumal, P.-A. Absil, and C.~Cartis.
\newblock Global rates of convergence for nonconvex optimization on manifolds.
\newblock \emph{IMA Journal of Numerical Analysis}, 39\penalty0 (1):\penalty0
  1--33, 2019.

\bibitem[Burke and Ferris(1993)]{burke1993weak}
J.~V. Burke and M.~C. Ferris.
\newblock Weak sharp minima in mathematical programming.
\newblock \emph{SIAM Journal on Control and Optimization}, 31\penalty0
  (5):\penalty0 1340--1359, 1993.

\bibitem[Chen et~al.(2020)Chen, Ma, Man-Cho~So, and Zhang]{chen2020proximal}
S.~Chen, S.~Ma, A.~Man-Cho~So, and T.~Zhang.
\newblock Proximal gradient method for nonsmooth optimization over the
  {S}tiefel manifold.
\newblock \emph{SIAM Journal on Optimization}, 30\penalty0 (1):\penalty0
  210--239, 2020.

\bibitem[Chen et~al.(2021{\natexlab{a}})Chen, Deng, Ma, and
  So]{chen2021manifold}
S.~Chen, Z.~Deng, S.~Ma, and A.~M.-C. So.
\newblock Manifold proximal point algorithms for dual principal component
  pursuit and orthogonal dictionary learning.
\newblock \emph{IEEE Transactions on Signal Processing}, 69:\penalty0
  4759--4773, 2021{\natexlab{a}}.

\bibitem[Chen et~al.(2021{\natexlab{b}})Chen, Garcia, Hong, and
  Shahrampour]{chen2021decentralized}
S.~Chen, A.~Garcia, M.~Hong, and S.~Shahrampour.
\newblock Decentralized {R}iemannian gradient descent on the {S}tiefel
  manifold.
\newblock \emph{arXiv preprint arXiv:2102.07091}, 2021{\natexlab{b}}.

\bibitem[Chen et~al.(2021{\natexlab{c}})Chen, Garcia, Hong, and
  Shahrampour]{chen2021local}
S.~Chen, A.~Garcia, M.~Hong, and S.~Shahrampour.
\newblock On the local linear rate of consensus on the {S}tiefel manifold.
\newblock \emph{arXiv preprint arXiv:2101.09346}, 2021{\natexlab{c}}.

\bibitem[Chen et~al.(2021{\natexlab{d}})Chen, Garcia, and
  Shahrampour]{chen2021distributed}
S.~Chen, A.~Garcia, and S.~Shahrampour.
\newblock On distributed non-convex optimization: Projected subgradient method
  for weakly convex problems in networks.
\newblock \emph{IEEE Transactions on Automatic Control}, 2021{\natexlab{d}}.

\bibitem[Clarke et~al.(1995)Clarke, Stern, and Wolenski]{clarke1995proximal}
F.~H. Clarke, R.~Stern, and P.~Wolenski.
\newblock Proximal smoothness and the lower-{$C^2$} property.
\newblock \emph{J. Convex Anal}, 2\penalty0 (1-2):\penalty0 117--144, 1995.

\bibitem[Daneshmand et~al.(2020)Daneshmand, Scutari, and
  Kungurtsev]{daneshmand2020second}
A.~Daneshmand, G.~Scutari, and V.~Kungurtsev.
\newblock Second-order guarantees of distributed gradient algorithms.
\newblock \emph{SIAM Journal on Optimization}, 30\penalty0 (4):\penalty0
  3029--3068, 2020.

\bibitem[Davis et~al.(2018)Davis, Drusvyatskiy, MacPhee, and
  Paquette]{davis2018subgradient}
D.~Davis, D.~Drusvyatskiy, K.~J. MacPhee, and C.~Paquette.
\newblock Subgradient methods for sharp weakly convex functions.
\newblock \emph{Journal of Optimization Theory and Applications}, 179\penalty0
  (3):\penalty0 962--982, 2018.

\bibitem[Davis et~al.(2020)Davis, Drusvyatskiy, and Shi]{davis2020stochastic}
D.~Davis, D.~Drusvyatskiy, and Z.~Shi.
\newblock Stochastic optimization over proximally smooth sets.
\newblock \emph{arXiv preprint arXiv:2002.06309}, 2020.

\bibitem[Duchi et~al.(2011)Duchi, Agarwal, and Wainwright]{duchi2011dual}
J.~C. Duchi, A.~Agarwal, and M.~J. Wainwright.
\newblock Dual averaging for distributed optimization: Convergence analysis and
  network scaling.
\newblock \emph{IEEE Transactions on Automatic Control}, 57\penalty0
  (3):\penalty0 592--606, 2011.

\bibitem[Hu et~al.(2020)Hu, Liu, Wen, and Yuan]{hu2020brief}
J.~Hu, X.~Liu, Z.-W. Wen, and Y.-X. Yuan.
\newblock A brief introduction to manifold optimization.
\newblock \emph{Journal of the Operations Research Society of China},
  8\penalty0 (2):\penalty0 199--248, 2020.

\bibitem[Huang et~al.(2021)Huang, Li, Milzarek, Pu, and
  Qiu]{huang2021distributed}
K.~Huang, X.~Li, A.~Milzarek, S.~Pu, and J.~Qiu.
\newblock Distributed random reshuffling over networks.
\newblock \emph{arXiv preprint arXiv:2112.15287}, 2021.

\bibitem[Karkhaneei and Mahdavi-Amiri(2019)]{karkhaneei2019nonconvex}
M.~M. Karkhaneei and N.~Mahdavi-Amiri.
\newblock Nonconvex weak sharp minima on {R}iemannian manifolds.
\newblock \emph{Journal of Optimization Theory and Applications}, 183\penalty0
  (1):\penalty0 85--104, 2019.

\bibitem[Li et~al.(2011)Li, Mordukhovich, Wang, and Yao]{li2011weak}
C.~Li, B.~S. Mordukhovich, J.~Wang, and J.-C. Yao.
\newblock Weak sharp minima on {R}iemannian manifolds.
\newblock \emph{SIAM Journal on Optimization}, 21\penalty0 (4):\penalty0
  1523--1560, 2011.

\bibitem[Li et~al.(2020{\natexlab{a}})Li, So, and Ma]{li2020understanding}
J.~Li, A.~M.-C. So, and W.-K. Ma.
\newblock Understanding notions of stationarity in nonsmooth optimization: A
  guided tour of various constructions of subdifferential for nonsmooth
  functions.
\newblock \emph{IEEE Signal Processing Magazine}, 37\penalty0 (5):\penalty0
  18--31, 2020{\natexlab{a}}.

\bibitem[Li et~al.(2020{\natexlab{b}})Li, Zhu, Man-Cho~So, and
  Vidal]{li2020nonconvex}
X.~Li, Z.~Zhu, A.~Man-Cho~So, and R.~Vidal.
\newblock Nonconvex robust low-rank matrix recovery.
\newblock \emph{SIAM Journal on Optimization}, 30\penalty0 (1):\penalty0
  660--686, 2020{\natexlab{b}}.

\bibitem[Li et~al.(2021{\natexlab{a}})Li, Chen, Deng, Qu, Zhu, and
  Man-Cho~So]{li2021weakly}
X.~Li, S.~Chen, Z.~Deng, Q.~Qu, Z.~Zhu, and A.~Man-Cho~So.
\newblock Weakly convex optimization over {S}tiefel manifold using {R}iemannian
  subgradient-type methods.
\newblock \emph{SIAM Journal on Optimization}, 31\penalty0 (3):\penalty0
  1605--1634, 2021{\natexlab{a}}.

\bibitem[Li et~al.(2021{\natexlab{b}})Li, Milzarek, and Qiu]{li2021convergence}
X.~Li, A.~Milzarek, and J.~Qiu.
\newblock Convergence of random reshuffling under the
  {K}urdyka-{$\L$}ojasiewicz inequality.
\newblock \emph{arXiv preprint arXiv:2110.04926}, 2021{\natexlab{b}}.

\bibitem[Liu et~al.(2022)Liu, Zhou, Pei, Zhang, and Shi]{liu2022decentralized}
C.~Liu, Z.~Zhou, J.~Pei, Y.~Zhang, and Y.~Shi.
\newblock Decentralized composite optimization in stochastic networks: A dual
  averaging approach with linear convergence.
\newblock \emph{IEEE Transactions on Automatic Control}, 2022.

\bibitem[Liu et~al.(2017)Liu, Yue, and Man-Cho~So]{liu2017estimation}
H.~Liu, M.-C. Yue, and A.~Man-Cho~So.
\newblock On the estimation performance and convergence rate of the generalized
  power method for phase synchronization.
\newblock \emph{SIAM Journal on Optimization}, 27\penalty0 (4):\penalty0
  2426--2446, 2017.

\bibitem[Liu et~al.(2019)Liu, So, and Wu]{liu2019quadratic}
H.~Liu, A.~M.-C. So, and W.~Wu.
\newblock Quadratic optimization with orthogonality constraint: explicit
  {$\L$}ojasiewicz exponent and linear convergence of retraction-based
  line-search and stochastic variance-reduced gradient methods.
\newblock \emph{Mathematical Programming}, 178\penalty0 (1):\penalty0 215--262,
  2019.

\bibitem[Liu et~al.(2020{\natexlab{a}})Liu, Deng, Li, Chen, and
  So]{liu2020nonconvex}
H.~Liu, Z.~Deng, X.~Li, S.~Chen, and A.~M.-C. So.
\newblock Nonconvex robust synchronization of rotations.
\newblock In \emph{NeurIPS Annual Workshop on Optimization for Machine
  Learning}, pages 1--7, 2020{\natexlab{a}}.

\bibitem[Liu et~al.(2020{\natexlab{b}})Liu, Yue, and So]{liu2022unified}
H.~Liu, M.-C. Yue, and A.~M.-C. So.
\newblock A unified approach to synchronization problems over subgroups of the
  orthogonal group.
\newblock \emph{arXiv preprint arXiv:2009.07514}, 2020{\natexlab{b}}.

\bibitem[Luss and Teboulle(2013)]{luss2013conditional}
R.~Luss and M.~Teboulle.
\newblock Conditional gradient algorithms for rank-one matrix approximations
  with a sparsity constraint.
\newblock \emph{SIAM Review}, 55\penalty0 (1):\penalty0 65--98, 2013.

\bibitem[Markdahl et~al.(2020)Markdahl, Thunberg, and
  Goncalves]{markdahl2020high}
J.~Markdahl, J.~Thunberg, and J.~Goncalves.
\newblock High-dimensional {K}uramoto models on {S}tiefel manifolds synchronize
  complex networks almost globally.
\newblock \emph{Automatica}, 113:\penalty0 108736, 2020.

\bibitem[Moreau(1965)]{moreau1965proximite}
J.-J. Moreau.
\newblock Proximit{\'e} et dualit{\'e} dans un espace hilbertien.
\newblock \emph{Bulletin de la Soci{\'e}t{\'e} math{\'e}matique de France},
  93:\penalty0 273--299, 1965.

\bibitem[Nedic et~al.(2010)Nedic, Ozdaglar, and Parrilo]{nedic2010constrained}
A.~Nedic, A.~Ozdaglar, and P.~A. Parrilo.
\newblock Constrained consensus and optimization in multi-agent networks.
\newblock \emph{IEEE Transactions on Automatic Control}, 55\penalty0
  (4):\penalty0 922--938, 2010.

\bibitem[Pillai et~al.(2005)Pillai, Suel, and Cha]{pillai2005perron}
S.~U. Pillai, T.~Suel, and S.~Cha.
\newblock The perron-frobenius theorem: some of its applications.
\newblock \emph{IEEE Signal Processing Magazine}, 22\penalty0 (2):\penalty0
  62--75, 2005.

\bibitem[Polyak(1987)]{polyak1987introduction}
B.~T. Polyak.
\newblock Introduction to optimization. optimization software.
\newblock \emph{Inc., Publications Division, New York}, 1, 1987.

\bibitem[Rockafellar and Wets(2009)]{rockafellar2009variational}
R.~T. Rockafellar and R.~J.-B. Wets.
\newblock \emph{Variational Analysis}, volume 317.
\newblock Springer Science \& Business Media, 2009.

\bibitem[So and Zhou(2017)]{so2017non}
A.~M.-C. So and Z.~Zhou.
\newblock Non-asymptotic convergence analysis of inexact gradient methods for
  machine learning without strong convexity.
\newblock \emph{Optimization Methods and Software}, 32\penalty0 (4):\penalty0
  963--992, 2017.

\bibitem[Tian et~al.(2019)Tian, Sun, and Scutari]{tian2019asynchronous}
Y.~Tian, Y.~Sun, and G.~Scutari.
\newblock Asynchronous decentralized successive convex approximation.
\newblock \emph{arXiv preprint arXiv:1909.10144}, 2019.

\bibitem[Tsitsiklis et~al.(1986)Tsitsiklis, Bertsekas, and
  Athans]{tsitsiklis1986distributed}
J.~Tsitsiklis, D.~Bertsekas, and M.~Athans.
\newblock Distributed asynchronous deterministic and stochastic gradient
  optimization algorithms.
\newblock \emph{IEEE Transactions on Automatic Control}, 31\penalty0
  (9):\penalty0 803--812, 1986.

\bibitem[Wang and Liu(2022)]{wang2022decentralized}
L.~Wang and X.~Liu.
\newblock Decentralized optimization over the {S}tiefel manifold by an
  approximate augmented {L}agrangian function.
\newblock \emph{IEEE Transactions on Signal Processing}, 70:\penalty0
  3029--3041, 2022.

\bibitem[Wang et~al.(2021)Wang, Liu, and So]{wang2021linear}
P.~Wang, H.~Liu, and A.~M.-C. So.
\newblock Linear convergence of a proximal alternating minimization method with
  extrapolation for $\ell_1$-norm principal component analysis.
\newblock \emph{arXiv preprint arXiv:2107.07107}, 2021.

\bibitem[Wang et~al.(2020)Wang, Wu, and Yu]{wang2020unique}
Y.~Wang, S.~Wu, and B.~Yu.
\newblock Unique sharp local minimum in $\ell_1$-minimization complete
  dictionary learning.
\newblock \emph{Journal of Machine Learning Research}, 21:\penalty0 63--1,
  2020.

\bibitem[Wu et~al.(2018)Wu, Wai, Li, and Scaglione]{wu2018review}
S.~X. Wu, H.-T. Wai, L.~Li, and A.~Scaglione.
\newblock A review of distributed algorithms for principal component analysis.
\newblock \emph{Proceedings of the IEEE}, 106\penalty0 (8):\penalty0
  1321--1340, 2018.

\bibitem[Yang et~al.(2014)Yang, Zhang, and Song]{yang2014optimality}
W.~H. Yang, L.-H. Zhang, and R.~Song.
\newblock Optimality conditions for the nonlinear programming problems on
  {R}iemannian manifolds.
\newblock \emph{Pacific Journal of Optimization}, 10\penalty0 (2):\penalty0
  415--434, 2014.

\bibitem[Ye and Zhang(2021)]{ye2021deepca}
H.~Ye and T.~Zhang.
\newblock Deepca: Decentralized exact {PCA} with linear convergence rate.
\newblock \emph{Journal of Machine Learning Research}, 22\penalty0
  (238):\penalty0 1--27, 2021.

\bibitem[Yuan et~al.(2016)Yuan, Ling, and Yin]{yuan2016convergence}
K.~Yuan, Q.~Ling, and W.~Yin.
\newblock On the convergence of decentralized gradient descent.
\newblock \emph{SIAM Journal on Optimization}, 26\penalty0 (3):\penalty0
  1835--1854, 2016.

\bibitem[Yue et~al.(2019)Yue, Zhou, and So]{yue2019family}
M.-C. Yue, Z.~Zhou, and A.~M.-C. So.
\newblock A family of inexact {SQA} methods for non-smooth convex minimization
  with provable convergence guarantees based on the {L}uo--{T}seng error bound
  property.
\newblock \emph{Mathematical Programming}, 174\penalty0 (1):\penalty0 327--358,
  2019.

\bibitem[Zeng and Yin(2018)]{zeng2018nonconvex}
J.~Zeng and W.~Yin.
\newblock On nonconvex decentralized gradient descent.
\newblock \emph{IEEE Transactions on Signal Processing}, 66\penalty0
  (11):\penalty0 2834--2848, 2018.

\bibitem[Zhou and So(2017)]{zhou2017unified}
Z.~Zhou and A.~M.-C. So.
\newblock A unified approach to error bounds for structured convex optimization
  problems.
\newblock \emph{Mathematical Programming}, 165\penalty0 (2):\penalty0 689--728,
  2017.

\bibitem[Zhu et~al.(2021)Zhu, Wang, and So]{zhu2021orthogonal}
L.~Zhu, J.~Wang, and A.~M.-C. So.
\newblock Orthogonal group synchronization with incomplete measurements: Error
  bounds and linear convergence of the generalized power method.
\newblock \emph{arXiv preprint arXiv:2112.06556}, 2021.

\bibitem[Zhu et~al.(2018)Zhu, Wang, Robinson, Naiman, Vidal, and
  Tsakiris]{zhu2018dual}
Z.~Zhu, Y.~Wang, D.~Robinson, D.~Naiman, R.~Vidal, and M.~Tsakiris.
\newblock Dual principal component pursuit: Improved analysis and efficient
  algorithms.
\newblock \emph{Advances in Neural Information Processing Systems}, 31, 2018.

\end{thebibliography}

\end{document}